\documentclass[12pt,reqno,a4]{amsart}
\usepackage{a4wide}
\usepackage{amsmath,amscd}
\usepackage{amsmath}
\usepackage{amsfonts}
\usepackage{amsthm}
\usepackage{amssymb}
\usepackage{enumitem}
\usepackage{mathrsfs}
\usepackage{hyperref}
\usepackage{float}
\usepackage[x11names, svgnames,dvipsnames, rgb,table]{xcolor}
\usepackage{tikz}
\usetikzlibrary{decorations.markings}
\usetikzlibrary{matrix}
\usepgflibrary{snakes,arrows,shapes}
\usetikzlibrary{graphs}

\theoremstyle{plain}
\newtheorem{thm}{Theorem}[section]
\newtheorem{cor}[thm]{Corollary}
\newtheorem{lem}[thm]{Lemma}
\newtheorem{prop}[thm]{Proposition}
\newtheorem*{claim*}{Claim}
\theoremstyle{definition}
\newtheorem{defn}[thm]{Definition}

\newtheorem{rem}[thm]{Remark}

\newtheorem{op}[thm]{Open Problem}

\newcommand{\ar}{{\mathcal{R}}}
\newcommand{\el}{{\mathcal{L}}}
\newcommand{\eh}{{\mathcal{H}}}
\newcommand{\dee}{{\mathcal{D}}}
\newcommand{\jay}{{\mathcal{J}}}

\renewcommand{\l}{\mbox{${\langle}$}}
\renewcommand{\r}{\mbox{${\rangle}$}}
\renewcommand{\a}{\alpha}
\renewcommand{\b}{\beta}
\newcommand{\g}{\gamma}
\renewcommand{\d}{\delta}
\newcommand{\lam}{\lambda}
\newcommand{\s}{\sigma}
\newcommand{\Sig}{\Sigma}
\renewcommand{\th}{\theta}
\renewcommand{\k}{\kappa}
\newcommand{\w}{\omega}
\newcommand{\al}{\aleph_0}
\newcommand{\ha}{\widehat{\alpha}}
\newcommand{\hb}{\widehat{\beta}}
\newcommand{\hg}{\widehat{\gamma}}
\newcommand{\ta}{\widetilde{\alpha}}
\newcommand{\tb}{\widetilde{\beta}}
\newcommand{\tg}{\widetilde{\gamma}}
\newcommand{\tnu}{\widetilde{\nu}}

\renewcommand{\phi}{\varphi}
\newcommand{\N}{\mathbb{N}}
\DeclareMathOperator{\Sym}{\mathcal{S}\hspace{-0.1em}}
\DeclareMathOperator{\PT}{\mathcal{PT}\!}
\DeclareMathOperator{\T}{\mathcal{T}\!}
\DeclareMathOperator{\B}{\mathcal{B}}
\renewcommand{\P}{\mathcal{P}}
\DeclareMathOperator{\PB}{\mathcal{PB}}
\DeclareMathOperator{\I}{\mathcal{I}\hspace{-0.1em}}
\DeclareMathOperator{\F}{\mathcal{F}\!}
\renewcommand{\H}{\mathcal{H}}

\DeclareMathOperator{\Inj}{\mathcal{I}\hspace{-0.1em}\mathit{nj}\hspace{-0.1em}}
\DeclareMathOperator{\BL}{\mathcal{BL}}
\DeclareMathOperator{\Surj}{\mathcal{S}\hspace{-0.1em}\mathit{urj}\hspace{-0.1em}}
\DeclareMathOperator{\DBL}{\mathcal{DBL}}

\DeclareMathOperator{\im}{im}
\DeclareMathOperator{\dom}{dom}

\newcounter{ncols}
\newcounter{incols}

\newcommand\nc\newcommand
\nc\rnc\renewcommand

\nc\AND{\qquad\text{and}\qquad}
\nc\ANd{\quad\text{and}\quad}
\nc\COMMA{,\qquad}
\nc\COMMa{,\quad}
\nc\WHERE{\qquad\text{where}\qquad}

\rnc\iff{\ \Leftrightarrow\ }
\nc\IFf{\quad \Leftrightarrow\quad }
\nc\Iff{\ \ \Leftrightarrow\ \ }
\nc\IFF{\qquad \Leftrightarrow\qquad }
\rnc\implies{\ \Rightarrow\ }
\nc\IMPLIES{\qquad \Rightarrow\qquad }

\newcommand{\multwk}[1]{\cdot}

\nc\permdiff\partial
\nc\pd\permdiff
\nc\Proj[1]{\overline{#1}}
\nc\norm[1]{\langle\!\langle#1\rangle\!\rangle}
\nc\fd\De

\rnc\emptyset{\varnothing}
\nc\sub\subseteq
\nc\mt\mapsto
\nc\wh\widehat
\nc\normal\unlhd
\nc\sm\setminus
\nc\mr\mathrel
\nc\pc[2]{(#1,#2)^\sharp}

\newcommand{\uv}[1]{\fill (#1,2)circle(.17);}
\newcommand{\lv}[1]{\fill (#1,0)circle(.17);}
\newcommand{\uvs}[1]{{\foreach \x in {#1} { \uv{\x}}}}
\newcommand{\lvs}[1]{{\foreach \x in {#1} { \lv{\x}}}}
\newcommand{\darcx}[3]{\draw(#1,0)arc(180:90:#3) (#1+#3,#3)--(#2-#3,#3) (#2-#3,#3) arc(90:0:#3);}
\newcommand{\darc}[2]{\darcx{#1}{#2}{.4}}
\newcommand{\uarcx}[3]{\draw(#1,2)arc(180:270:#3) (#1+#3,2-#3)--(#2-#3,2-#3) (#2-#3,2-#3) arc(270:360:#3);}
\newcommand{\uarc}[2]{\uarcx{#1}{#2}{.4}}
\newcommand{\stline}[2]{\draw(#1,2)--(#2,0);}
\nc{\buv}[1]{\fill (#1,2)circle(.18);}
\nc{\buvs}[1]{{
\foreach \x in {#1}
{ \buv{\x}}
}}
\nc{\blv}[1]{\fill (#1,0)circle(.18);}
\nc{\blvs}[1]{{
\foreach \x in {#1}
{ \blv{\x}}
}}
\nc{\uarcs}[1]{
{\foreach \x/\y in {#1}
{ \uarc{\x}{\y} }
}
}
\nc{\darcs}[1]{
{\foreach \x/\y in {#1}
{ \darc{\x}{\y} }
}
}
\nc{\darcxhalf}[3]{\draw(#1,0)arc(180:90:#3) (#1+#3,#3)--(#2,#3) ;}
\nc{\darchalf}[2]{\darcxhalf{#1}{#2}{.4}}
\nc{\uarcxhalf}[3]{\draw(#1,2)arc(180:270:#3) (#1+#3,1.5-#3)--(#2,1.5-#3) ;}
\nc{\uarchalf}[2]{\uarcxhalf{#1}{#2}{.4}}
\nc{\colv}[3]{\fill[#3] (#1,#2)circle(.17);}
\nc{\uvert}[1]{\fill (#1,2)circle(.2);}
\rnc{\lvert}[1]{\fill (#1,0)circle(.2);}
\nc{\custpartn}[3]{{\lower1.4 ex\hbox{
\begin{tikzpicture}[scale=.3]
\foreach \x in {#1}
{ \uvert{\x}  }
\foreach \x in {#2}
{ \lvert{\x}  }
#3 \end{tikzpicture}
}}}

\begin{document}
\title{\vspace{-3em}On the diameter of semigroups of transformations and partitions}
\author{James East, Victoria Gould, Craig Miller, Thomas Quinn-Gregson, Nik Ru{\v s}kuc}
\address{Centre for Reseach in Mathematics and Data Science, Western Sydney University, Australia}
\email{J.East@WesternSydney.edu.au}
\address{Department of Mathematics, University of York, UK, YO10 5DD}
\email{victoria.gould@york.ac.uk, craig.miller@york.ac.uk}
\address{~\vspace{-1em}}
\email{tquinngregson@gmail.com}
\address{School of Mathematics and Statistics, St Andrews, Scotland, UK, KY16 9SS}
\email{nik.ruskuc@st-andrews.ac.uk}
\maketitle

\vspace{-1em}
\begin{abstract}
For a semigroup $S$ whose universal right congruence is finitely generated (or, equivalently, a semigroup satisfying the homological finiteness property of being type right-$FP_1$), the right diameter of $S$ is a parameter that expresses how `far apart' elements of $S$ can be from each other, in a certain sense.  To be more precise, for each finite generating set $U$ for the universal right congruence on $S,$ we have a metric space $(S,d_U)$ where $d_U(a,b)$ is the minimum length of derivations for $(a,b)$ as a consequence of pairs in $U$; the right diameter of $S$ with respect to $U$ is the diameter of this metric space.  The right diameter of $S$ is then the minimum of the set of all right diameters with respect to finite generating sets.

We develop a theoretical framework for establishing whether a semigroup of transformations or partitions on an arbitrary infinite set $X$ has a finitely generated universal right/left congruence, and, if it does, determining its right/left diameter.  We apply this to prove results such as the following.  Each of the monoids of all binary relations on $X,$ of all partial transformations on $X,$ and of all full transformations on $X,$ as well as the partition and partial Brauer monoids on $X,$ have right diameter 1 and left diameter 1.  The symmetric inverse monoid on $X$ has right diameter 2 and left diameter 2.  The monoid of all injective mappings on $X$ has right diameter 4, and its minimal ideal (called the Baer-Levi semigroup on $X$) has right diameter 3, but neither of these two semigroups has a finitely generated universal left congruence.  On the other hand, the semigroup of all surjective mappings on $X$ has left diameter 4, and its minimal ideal has left diameter 2, but neither of these semigroups has a finitely generated universal right congruence.
\end{abstract}
~
\maketitle

\section{Introduction}\label{sec:intro}

This paper is concerned with the semigroup finiteness condition of the universal right congruence being finitely generated, and the related parameter of right diameter, as well as the left-right duals of these notions.

For a semigroup $S$ whose universal right congruence is generated by a finite set $U,$ the right diameter of $S$ with respect to $U$ is, informally, the supremum of the minimum lengths of derivations for pairs $(a,b)\in S\times S$ as a consequence of those in $U.$  The right diameter of $S$ is the minimum of the set of all right diameters with respect to finite generating sets.  Thus, a semigroup has finite right diameter if its universal right congruence is finitely generated and there is a bound on the length of sequences required to relate any two elements.  More precise definitions regarding the notion of diameter will be given in Section \ref{sec:prelim}.

The property of having finite right (resp.\ left) diameter is also known as being {\em right} (resp.\ {\em left}) {\em pseudo-finite}.  Left pseudo-finite semigroups were first studied by White in \cite{White:2017} in the context of Banach algebras.  This work was motivated by a conjecture of Dales and {\.Z}elazko, stating that a unital Banach algebra in which every maximal left ideal is finitely generated is necessarily finite dimensional.  
It was also noted in \cite{White:2017} that for weakly right cancellative monoids, which include groups, being left pseudo-finite coincides with being finite.

Dandan et al.\ undertook the first comprehensive study of semigroups with a finitely generated universal left congruence, with appropriate specialisions to left pseudo-finite semigroups \cite{Dandan}.  The former class of semigroups was shown to be equivalent to a number of previously-studied classes, including those semigroups satisfying the homological finiteness property of being type left-$FP_1$ \cite[Theorem 3.10]{Dandan} (the equivalence of some of these conditions had previously been established by Kobayashi \cite{Kob}).  
An interesting question raised in \cite[Open Question 8.10]{Dandan} is whether every (left) pseudo-finite semigroup has a completely simple minimal ideal.  The article \cite{Gould} sought to address this question systematically.  It found that for pseudo-finite semigroups lying in some important classes, such as orthodox semigroups, completely regular semigroups and commutative semigroups, having a completely simple minimal ideal {\em is} necessary, but in general a pseudo-finite semigroup may have a minimal ideal that it not completely simple, or may have no minimal ideal at all.

The notion of right/left diameter was introduced in \cite{Gould} as a useful tool for proving that certain semigroups are right/left pseudo-finite.  
It was observed that the property of having right diameter 1 is equivalent to a certain well-studied notion, namely that of the diagonal right act being finitely generated \cite[Proposition 3.6]{Gould}.  For a semigroup $S,$ the diagonal right $S$-act is the set $S\times S$ under the right action given by $(a,b)c=(ac,bc).$
Diagonal acts first appear, implicitly, in \cite{Bulman:1989}, and they were then formally defined and studied in \cite{Robertson:2001}.  A systematic investigation into generation of diagonal acts was undertaken in \cite{Gallagher:thesis}, and some of the most intriguing results concerned certain infinite semigroups of transformations and relations \cite{Gallagher:2005}.  In particular, it was shown that, for any infinite set $X,$ the diagonal right and left acts are monogenic for the monoids $\B_X$ of all binary relations on $X,$ $\T_X$ of all transformations on $X,$ and $\PT_X$ of all partial transformations on $X.$

Given the above findings concerning certain transformation semigroups, it is natural to consider similar kinds of semigroups when searching for semigroups with finite right/left diameter.  Indeed, the first example found of a right pseudo-finite semigroup with a minimal ideal that is not completely simple was the Baer-Levi semigroup on an infinite set $X$ \cite[Remark 7.3]{Miller:2020}, and another such example is the monoid of all injective mappings on $X$ \cite[Proposition 4.4]{Gould}.  Moreover, the first example exhibited of a right (and left) pseudo-finite semigroup with {\em no} minimal ideal was a certain transformation monoid denoted $\mathcal{U}_X$ \cite[Example 8.1]{Gould}.

A class of semigroups that exhibit some similar behaviour to transformation semigroups is that of the so-called diagram monoids, which have recently come into prominence; see \cite{FitzGerald}.  In particular, the partition monoid on a set $X,$ denoted $\P_X$, contains natural copies of many `classical' transformation monoids, including the symmetric group $\Sym_X$, the full transformation monoid $\T_X$ and the symmetric inverse monoid $\I_X$.  The importance of these classical monoids derives mainly from the well-known Cayley Theorems, stating that every group embeds into some $\Sym_X$ and every semigroup into some $\T_X$ \cite[Theorem 1.1.2]{Howie}, and the Wagner-Preston Theorem, stating that every inverse semigroup embeds into some $\I_X$ \cite[Theorem 5.1.7]{Howie}.  Thus, a common theme in papers on partition monoids is the extent to which their behaviour resembles those of classical transformation monoids; for example, see the article \cite{East}, which classifies all congruences on $\P_X$ and the partial Brauer monoid $\PB_X$, where $X$ is an arbitrary infinite set.
Given the aforementioned results concerning certain classical transformation monoids in relation to diameter, it is natural to explore monoids of partitions as a potential source of further examples of semigroups with finite right/left diameter.

The purpose of this article is to systematically investigate, for various infinite semigroups of transformations and partitions, whether each such semigroup has a finitely generated universal right/left congruence, and, if so, determine its right/left diameter.  The main results are summarised in Table \ref{table:results}.

The paper is organised as follows.  In Section \ref{sec:prelim} we provide the necessary prelimary material and summarise the main results of the paper.  Various transformation semigroups are considered in Sections \ref{sec:right} and \ref{sec:left}.  Section \ref{sec:right} is concerned with the universal right congruence and right diameter, and Section \ref{sec:left} is the left-right counterpart of Section \ref{sec:right}.  In both these sections, we aim to prove general results that can be applied to a number of transformation semigroups of concern.  Particularly noteworthy results are obtained for certain subsemigroups of the monoids $\Inj_X$ of all injective mappings on $X$ and $\Surj_X$ of all surjective mappings $X.$  In particular, we prove that:
\begin{itemize}
\item the minimal ideal $\BL_X$ of $\Inj_X$, called the Baer-Levi semigroup on $X,$ has right diameter 3 (Theorem \ref{thm:BL});
\item a submonoid $S$ of $\Inj_X$ containing the symmetric group $\Sym_X$ has right diameter 4 if and only if it contains $\BL_X$ (Theorem \ref{thm:S,Inj});
\item the minimal ideal $\DBL_X$ of $\Surj_X$, called the dual Baer-Levi semigroup on $X,$ has left diameter 2 (Theorem \ref{thm:DBL});
\item the monoid $\Surj_X$ has left diameter 4 (Theorem \ref{thm:S,Surj}).
\end{itemize} 
Finally, in Section \ref{sec:partition} we prove that both the partition monoid $\P_X$ and the partial Brauer monoid $\PB_X$ have right diameter 1 and left diameter 1. 
It is perhaps intriguing that all diameters computed in this paper and shown in Table \ref{table:results} are `small', specifically $\leq 4$.
At this point we do not know any examples of `natural' semigroups of transformations or partitions whose diameters are finite and greater than $4$.

\section{Notation and Summary of Results}\label{sec:prelim}

In this section we provide the necessary preliminary material on semigroups and summarise the main results of the article.  We refer the reader to \cite{Howie} for a more comprehensive introduction to the basic semigroup concepts defined here.

\subsection{Diameter of semigroups}\label{subsec:diameter}

Let $S$ be a semigroup.  We denote by $S^1$ the monoid obtained from $S$ by adjoining an identity if necessary (if $S$ is already a monoid, then $S^1=S$).

A {\em right ideal} of $S$ is a subset $I$ such that $IS\subseteq I.$  A subset $U$ of a right ideal $I$ is a {\em generating set} for $I$ if $I=US^1;$ $I$ is said to be {\em finitely generated} if it has a finite generating set.  Of course, $S$ is a right ideal of itself.  When considering $S$ being generated by a set {\em as a right ideal}, we shall write $S$ as $S^r$.  Thus, `$S^r$ is generated by $U$' means $S=US^1.$

An equivalence relation $\rho$ on $S$ is a {\em right congruence} if $(a, b)\in\rho$ implies $(as, bs)\in\rho$ for all $s\in S$. 
For $U\subseteq S\times S,$ the {\em right congruence generated by} $U$ is the smallest right congruence on $S$ containing $U$; we denote this right congruence by $\l U\r.$  

\begin{lem}\label{lem:sequence}\cite[Lemma I.\,4.\,37]{kkm}
Let $S$ be a semigroup, let $U$ be a subset of $S\times S,$ and let $\rho$ be the right congruence generated by $U.$  For any $a, b\in S,$ we have $(a,b)\in\rho$ if and only if either $a=b$ or there exists a sequence $$a=u_1s_1,\ v_1s_1=u_2s_2,\ \dots,\ v_ns_n=b$$ for some $n\in\N,$ where $(u_i,v_i)\in U$ or $(v_i,u_i)\in U,$ and $s_i\in S^1,$ for each $i\in\{1,\dots,n\}.$
\end{lem}

A sequence of the form given in Lemma \ref{lem:sequence} is referred to as a {\em $U$-sequence from $a$ to $b$ of length $n$}.  
If $a=b,$ we consider that $a$ and $b$ are related by a $U$-sequence of length 0.  
If the generating set $U$ consists of a single pair $(u,v),$ we may speak of $(u,v)$-sequences rather than $U$-sequences.

The universal relation $\w_S=S\times S$ is certainly a right congruence on $S.$  When viewing this relation as a right congruence, we shall denote it by $\w_S^r$.  If $U$ is a generating set for $\w_S^r$, we shall write $\w_S^r=\l U\r.$

Consider a set $U\subseteq S\times S$ such that $\w_S^r=\l U\r$.
For any $a, b\in S,$ let $d_U^r(a,b)$ denote the least non-negative integer $n$ such that there is a $U$-sequence of length $n$ from $a$ to $b.$ 
It is easy to see that $d_U^r : S\times S\to\{0,1,2,\dots\}$ is a metric on $S.$

\begin{defn}
Let $S$ be a semigroup.
\begin{itemize}[leftmargin=*]
\item If $\w_S^r=\l U\r$, we call the diameter of the metric space $(S,d_U^r)$ the {\em right $U$-diameter} of $S$ and denote it by $D_r(U,S)$; that is, 
\[D_r(U,S)=\sup\{d_U^r(a,b)\::\: a, b\in S\}.\]
\item If $\w_S^r$ is finitely generated, we define the {\em right diameter} of $S$ to be
$$D_r(S)=\min\{D_r(U,S) : \w_S^r=\l U\r, |U|<\al\}.$$
\end{itemize}
\end{defn}

Note that if $U$ and $U'$ are two finite generating sets for $\w_S^r,$ then $D_r(U,S)$ is finite if and only if $D_r(U',S)$ is finite \cite[Lemma 2.5]{Dandan}.  We make the following easy observation.

\begin{lem}\label{lem:w}
Let $S$ be a non-trivial semigroup.  If $\w_S^r=\l U\r$ then, letting  
$$V=\{v\in S : \exists u\in S\text{ such that }(u,v)\in U\text{ or }(v,u)\in U\},$$
we have $\w_S^r=\l V\times V\r$ and $D_r(V\times V,S)\leq D_r(U,S).$  Furthermore, we have $S=VS^1$.  In particular, if $\w_S^r$ is finitely generated then so is $S^r$.
\end{lem}

We shall often abuse terminology by saying that $\w_S^r$ is generated by a subset $V$ of $S$ to mean that $\w_S^r$ is generated by $V\times V$, and also write $D_r(V,S)$ in place of $D_r(V\times V,S).$  It follows from Lemma \ref{lem:w} that if $\w_S^r$ is finitely generated then there exists a finite subset $V\subseteq S$ such that $D_r(S)=D_r(V,S).$

We now provide some results that will be useful later in the paper.  

\begin{lem}\label{lem:rightideal}
Let $S$ be a monoid and let $I$ be a right ideal of $S.$  If $\w_I^r$ is finitely generated, then $\w_S^r$ is finitely generated.  Moreover, we have $D_r(S)\leq D_r(I)+2.$
\end{lem}

\begin{proof}
This result essentially follows from the proof of \cite[Lemma 2.11]{Dandan}.  We provide a proof here for completeness.

Let $U\subseteq I$ be a finite generating set for $\w_I^r$ such that $D_r(U,I)=D_r(I)$.  Choose any $u\in U.$  For any $a,b\in S,$ since $ua,ub\in I,$ there exists a $U$-sequence
$$ua=u_1s_1,\ v_1s_1=u_2s_2,\ \dots,\ v_ks_k=ub$$
in $I,$ where $k\leq D_r(I).$  Thus, letting $V=U\cup\{1\},$ we have a $V$-sequence
$$a=1a,\ ua=u_1s_1,\ v_1s_1=u_2s_2,\ \dots,\ v_ks_k=ub,\ 1b=b$$
from $a$ to $b$ of length $k+2\leq D_r(I)+2.$  We conclude that $\w_S^r$ is generated by $V,$ and $D_r(S)\leq D_r(V,S)\leq D_r(I)+2.$
\end{proof}

\begin{cor}\label{cor:zero}
If $S$ is a monoid with a left zero, then $D_r(S)\leq 2.$
\end{cor}

Green's relations $\el$, $\ar$, $\eh$, $\dee$ and $\jay$ are standard tools for describing the ideal structure of a semigroup.  The relation $\el$ on $S$ is given by $(a,b)\in\el$ if and only if $S^1a=S^1b$, i.e.\ if $a$ and $b$ generate the same principal left ideal.  The relations $\ar$ and $\jay$ are defined analogously in terms of principal right ideals and principal two-sided ideals, respectively.  Finally, we have $\eh=\el\cap\ar$ and $\dee=\el\vee\ar~(=\el\circ\ar=\ar\circ\el).$  

We call $S$ {\em left/right simple} if it has a single $\el/\ar$-class, and {\em simple} if it has a single $\jay$-class.   There is a natural partial order on the set of $\jay$-classes of $S,$ given by $J_a\leq J_b$ if and only $S^1aS^1\subseteq S^1bS^1$.  There is at most one minimal $\jay$-class under this ordering; if it exists, it is called the {\em minimal ideal} of $S,$ and is a simple subsemigroup of $S.$  

The equivalence relation $\el^{\ast}$ on $S$ is defined by the rule that $(a,b)\in\el^{\ast}$ if and only $a,b$ are $\el$-related in some oversemigroup $T$, i.e.\ $aT^1=bT^1$.  We say that $S$ is {\em $\el^{\ast}$-simple} if it has a single $\el^{\ast}$-class.  We dually define the relation $\ar^{\ast}$ and the notion of being $\ar^{\ast}$-simple.  By \cite[Proposition 3.4]{Gould}, an $\el^{\ast}$-simple semigroup has finite right diameter if and only if it is finite.  We provide a proof of this result here using a more general argument, which also shows that, for any $\el^{\ast}$-simple semigroup, being countable is necessary for the universal right congruence to be finitely generated.  

\begin{prop}\label{prop:L*-simple}
Let $S$ be an $\el^{\ast}$-simple semigroup.  If $\w_S^r$ is finitely generated, then $S$ is countable.  Moreover, $S$ has finite right diameter if and only if it is finite.
\end{prop}

\begin{proof}
Since $S$ is $\el^{\ast}$-simple, by \cite[Theorem 1]{Pastijn:1975} there exists an oversemigroup $T$ of $S$ such that $S$ is contained in a single $\el$-class of $T.$
(One can take $T$ to be the dual of the full transformation monoid on $S^1,$ in which maps are composed from right to left.)

Now, suppose that $\w_S^r$ is finitely generated,
and let $U\subseteq S$ be a finite generating set for $\w_S^r$ such that $D_r(S,U)=D_r(S).$  For each pair $u, v\in U,$ since $u$ and $v$ are $\el$-related in $T$ we can choose $\a(u,v)\in T$ such that $u=\a(u,v)v.$  Fix $b\in S.$  Let
$$V=\{\a(u_1,v_1)\dots\a(u_k,v_k)b : u_i, v_i\in U, k\leq D_r(S)\}\subseteq T.$$
Clearly $V$ is countable, and if $D_r(S)$ is finite then so is $V.$  We claim that $S\subseteq V$; then $S$ is countable, and it is finite if it has finite right diameter.  So, let $a\in S.$  Then there exists a $U$-sequence
$$a=u_1s_1,\ v_1s_1=u_2s_2,\ \dots,\ v_ks_k=b$$
where $k\leq D_r(S).$
Letting $\a_i=\a(u_i,v_i),$ we have 
$$a=u_1s_1=\a_1v_1s_1=\a_1u_2s_2=\a_1\a_2v_2s_2=\cdots
=\a_1\dots\a_kv_ks_k=\a_1\dots\a_kb\in V,$$
as required.  
Clearly, if $S$ is finite then it has finite right diameter.
\end{proof}

\begin{rem}
Combining Proposition \ref{prop:L*-simple} and \cite[Proposition 2.7]{Dandan}, it follows that any finitely generated infinite group $G$ has the property that $\w_G^r$ is finitely generated, but $D_r(G)$ is infinite.
Also, a slight modification of the proof of Proposition \ref{prop:L*-simple} shows that an $\el^{\ast}$-simple semigroup $S$ is countable if and only if $\w_S^r$ is countably generated (where {\em countably generated} means being generated by a countable set).
\end{rem}

The above definitions and results have obvious left-right duals, and we use analagous nomenclature and notation: left ideal, $S^l$, $\w_S^l$, left diameter, etc.

\subsection{Semigroups of transformations and relations}\label{subsec:trans}

In this subsection we introduce the transformation semigroups of concern in this article.  First, we recall some basic terminology regarding relations and mappings.

Throughout the paper $X$ will stand for an arbitrary  infinite set.

A ({\em binary}) {\em relation} on $X$ is a subset of $X\times X$.  We denote the identity relation $\{(x,x) : x\in X\}$ by $1_X$.  For a relation $\a$ on $X$ and a subset $Y\subseteq X,$ we define $$Y\a=\{x\in X : (y,x)\in\a\text{ for some }y\in Y\},$$ 
and we abbreviate $\{y\}\a$ to $y\a.$  The {\em domain} and {\em image} of $\a$ are, respectively,
$$\dom\a=\{x\in X : (x,y)\in\a\text{ for some }y\in X\}\quad\text{and}\quad\im\a=X\a,$$
and the {\em inverse} of $\a$ is the relation
$$\a^{-1}=\{(y,x) : (x,y)\in\a\}.$$
The {\em composition} of two relations $\a$ and $\b$ on $X$ is the relation
$$\a\b=\{(x,y) : \exists z\in X\text{ such that }(x,z)\in\a\text{ and }(z,y)\in\b\}.$$
A {\em partial transformation} on $X$ is a relation $\a$ on $X$ satisfying the condition 
$$(x,y),(x,z)\in\a\Rightarrow y=z.$$
Let $\a$ be a partial transformation on $X.$  
For each $x\in\dom\a$, we interpret $x\a$ as an element of $X$ (rather than a singleton subset of $X$).  Note that for $Y\subseteq X$ we have $Y\a^{-1}=\{x\in X : x\a\in Y\}.$  The {\em kernel} of $\a$ is 
$$\ker\a=\{(x,y)\in\dom\a\times\dom\a : x\a=y\a\}.$$
Observe that $\a\a^{-1}=\ker\a$ and $\a^{-1}\a=\{(x,x) : x\in\im\a\}.$  It follows that
$$\a\a^{-1}=1_X\,\Leftrightarrow\,[\dom\a=X\text{ and }\a\text{ is injective}]\quad\text{and}\quad\a^{-1}\a=1_X\,\Leftrightarrow\,\a\text{ is surjective}.$$
We now define the semigroups of transformations and relations that will be considered in this paper, with some relevant additional information.
~\vspace{0.7em}

\noindent
{\bf \boldmath{$\B_X$}:}
the monoid of all binary relations on $X$ under composition, with identity~$1_X$.
\vspace{1mm}

\noindent
{\bf \boldmath{$\PT_X$} -- the partial transformation monoid on \boldmath{$X$}:}
the submonoid of $\B_X$ consisting of all partial transformations on $X.$
\vspace{1mm}

\noindent
{\bf \boldmath{$\I_X$} -- the symmetric inverse monoid on \boldmath{$X$}:}
the submonoid of $\PT_X$ consisting of all injective partial transformations (also known as {\em partial bijections}).
\vspace{1mm}

\noindent
{\bf \boldmath{$\T_X$} -- the full transformation monoid on \boldmath{$X$}:}
the submonoid of $\PT_X$ consisting of all (full) transformations on $X,$ i.e.\
$\T_X=\{\alpha\in\PT_X\::\: \dom\alpha=X\}$.
\vspace{1mm}

\noindent
{\bf \boldmath{$\Sym_X$} -- the symmetric group on \boldmath{$X$}:}
the subgroup of $\T_X$ consisting of all bijections.
\vspace{1mm}

\noindent
{\bf \boldmath{$\F_X$}:}
the submonoid of $\T_X$ consisting of all finite-to-one mappings, i.e.\
$\F_X=\{\alpha\in\T_X\::\: |x\alpha^{-1}|<\infty \text{ for all } x\in X\}$.
\vspace{1mm}

\noindent
{\bf \boldmath{$\Inj_X$}:}
the submonoid of $\T_X$ consisting of all injective mappings.
\begin{itemize}[leftmargin=6mm, label=-]
\item $\Inj_X$ is right cancellative (that is, $\b\a=\g\a$ implies that $\b=\g$), and hence $1_X$ is its only idempotent.
\end{itemize}
\vspace{1mm}

\noindent
{\bf \boldmath{$\BL_{X,q}$} -- the Baer-Levi semigroup of type \boldmath{$q$} on \boldmath{$X$}:}
for an infinite cardinal $q\leq |X|$, it is the subsemigroup of $\Inj_X$ defined by
\[
\BL_{X,q}=\{\a\in\Inj_X : |X\!\setminus\!\im\a|=q\}.
\]
\begin{itemize}[leftmargin=6mm, label=-]
\item Each $\BL_{X,q}$ is right cancellative, right simple and has no idempotents \cite[Theorem~8.2]{Clifford:1967}.  
\end{itemize}
\vspace{1mm}

\noindent
{\bf \boldmath{$\BL_{X}$} -- the Baer-Levi semigroup on \boldmath{$X$}:}
$\BL_{X,q}$ for $q=|X|$.
\begin{itemize}[leftmargin=6mm, label=-]
\item For any $\a,\b\in\Inj_X$, we have $\a\in\b(\Inj_X)$ if and only if $\a\in(\Inj_X)\b(\Inj_X)$ if and only if $|X\!\setminus\!\im\a|\geq|X\!\setminus\!\im\b|.$  Thus, the $\jay(=\!\ar)$-classes of $\Inj_X$ form a chain
$$\qquad\qquad\Sym_X>J_1>J_2>\cdots>J_n>\cdots>\BL_{X,\aleph_0}>\BL_{X,\aleph_1}>\cdots>\BL_{X,|X|}=\BL_X,$$
where $J_n=\{\a\in\Inj_X : |X\!\setminus\!\im\a|=n\}$ ($n\in\N$), and $\BL_X$ is the minimal ideal of $\Inj_X$ \cite[Proposition 2.2, Theorem 2.3, Remark 2.4]{Inj}.
\end{itemize}
\vspace{1mm}

\noindent
{\bf \boldmath{$\BL_{X}^1$}:} the Baer-Levi semigroup with an identity adjoined.
\vspace{1mm}

\noindent
{\bf \boldmath{$\Sym_X\cup \BL_{X}$}:} the submonoid of $\T_X$ consisting of all bijections and all Baer-Levi elements.
\begin{itemize}[leftmargin=6mm, label=-]
\item For any subgroup $G$ of $\Sym_X$, the set $G\cup\BL_X$ is a submonoid of $\Inj_X$; see \cite[Exercise 8.1.10]{Clifford:1967} for more information about such monoids.
\end{itemize}
\vspace{1mm}

\noindent
{\bf \boldmath{$\Surj_X$}:} the submonoid of $\T_X$ consisting of all surjective mappings.
\begin{itemize}[leftmargin=6mm, label=-]
\item $\Surj_X$ is left cancellative, and hence $1_X$ is its only idempotent. 
\end{itemize}
\vspace{1mm}

\noindent
{\bf \boldmath{$\DBL_{X,q}$} -- the dual Baer-Levi semigroup of type \boldmath{$q$} on \boldmath{$X$}:}
for an infinite cardinal $q\leq |X|$, it is the subsemigroup of $\Surj_X$ defined by
\[
\DBL_{X,q}=\{\a\in\Surj_X : |x\alpha^{-1}|=q\text{ for all } x\in X\}.
\]
\begin{itemize}[leftmargin=6mm, label=-]
\item Each $\DBL_{X,q}$ is left cancellative, left simple and has no idempotents  \cite[Theorem 3]{Chen}.  
\end{itemize}
\vspace{1mm}

\noindent
{\bf \boldmath{$\DBL_{X}$} -- the dual Baer-Levi semigroup on \boldmath{$X$}:}
$\DBL_{X,q}$ for $q=|X|$.
\begin{itemize}[leftmargin=6mm, label=-]
\item 
$\DBL_X$ is the minimal ideal of $\Surj_X$ \cite[Theorem 3.2]{Surj}.
\end{itemize}
\vspace{1mm}

\noindent
{\bf \boldmath{$\DBL_{X}^1$}:} the dual Baer-Levi semigroup with an identity adjoined.
\vspace{1mm}

\noindent
{\bf \boldmath{$\Sym_X\cup \DBL_{X}$}:} the submonoid of $\T_X$ consisting of all bijections and all dual Baer-Levi elements.
\vspace{1mm}

\noindent
{\bf \boldmath{$\T_X\!\setminus\!\Inj_{X}$}:} the subsemigroup of $\T_X$ consisting of all non-injective mappings.
\vspace{1mm}

\noindent
{\bf \boldmath{$\T_X\!\setminus\!\Surj_{X}$}:} the subsemigroup of $\T_X$ consisting of all non-surjective mappings.
\vspace{1mm}

\noindent
{\bf \boldmath{$\H_X$}:}
the submonoid of $\T_X$ defined by
\[\H_X=\{\a\in\T_X : |Y\a|=|X|\text{ for all }Y\subseteq X\text{ with }|Y|=|X|\}.\]
\begin{itemize}[leftmargin=6mm, label=-]
\item $\H_X$ is {\em bisimple}, meaning that it has a single $\dee$-class.  It was introduced by Higgins in \cite{Higgins} as a means of proving that every semigroup embeds into some bisimple monoid.
\item The following are subsemigroups of $\H_X$: $\F_X$; $\Inj_X$ (and hence $\Sym_X$, $\BL_{X,q}$ where $\al\leq q\leq|X|,$ $\BL_X^1$, and $\Sym_X\cup\BL_X$); and $\DBL_{X,q}$ where $\al\leq q<|X|.$  This is clear in the case of $\Inj_X$; for the other semigroups we provide a brief explanation.  Suppose that $S$ is either $\F_X$ or $\DBL_{X,q}$ with $q<|X|$, and consider $\a\in S$ and $Y\subset X$ such that $|Y\a|<|X|.$  By definition, $|x\a^{-1}|\leq q$ for all $x\in X$ (in fact, $|x\a^{-1}|<\al$ if $S=\F_X$).  Therefore, using the fact that $Y\subseteq(Y\a)\a^{-1}$, we have 
$$\qquad\qquad|Y|\leq|(Y\a)\a^{-1}|=\Big|\bigcup_{x\in Y\a}x\a^{-1}\Big|\leq\max(|Y\a|,q)<|X|.$$
\end{itemize}

\noindent
{\bf \boldmath{$\mathcal{K}_X$}:} the submonoid of $\H_X$ defined by 
$$\mathcal{K}_X=\{\a\in\H_X : |X\!\setminus\!\im\a|<|X|\}.$$

All the semigroups in the above list are subsemigroups of $\T_X$, with the exception of $\B_X$, $\PT_X$ and $\I_X$.  All these subsemigroups of $\T_X$ have the following `transitivity' properties, which will play a key role in the paper.

\begin{defn}
Let $S$ be a subsemigroup of $\T_X.$ 
\begin{itemize}[leftmargin=*]
\item Let $\k\leq|X|$ be a cardinal.  We say that $S$ is {\em $\k$-transitive} if for any partial bijection $\lam\in\I_X$ with $|\!\dom\lam|\leq\k$ and $|X\!\setminus\!\dom\lam|=|X\!\setminus\!\im\lam|=|X|,$ there exists some $\th\in S$ extending $\lam,$ i.e.\ $\th|_{\dom\lam}=\lam.$
\item We call $S$ {\em finitely transitive} if it is $\k$-transitive for every finite cardinal $\k<\al$.
\end{itemize}
\end{defn}
~
\vspace{-0.5em}
\begin{rem} 
Let $S$ be a subsemigroup of $\T_X$, and let $\k\leq|X|.$ 
\begin{enumerate}[leftmargin=*]
\item The semigroup $S$ is $\k$-transitive if and only if it is $\mu$-transitive for every $\mu\leq\k.$
\item If $\k<|X|$, then $S$ is $\k$-transitive if and only if for any partial bijection $\lam\in\I_X$ with $|\!\dom\lam|\leq\k$ there exists some $\th\in S$ extending $\lam.$
\item If $S$ contains a $\k$-transitive (resp.\ finitely transitive) subsemigroup $T,$ then $S$ is also $\k$-transitive (resp.\ finitely transitive).
\end{enumerate}
\end{rem}

\subsection{Summary of results, and diagonal acts}\label{subsec:summary}

Our main goal in this paper is to answer the following questions for each semigroup $S$ listed in Section \ref{subsec:trans}, as well as the partition monoid $\P_X$ and the partial Brauer monoid $\PB_X$, which will be defined in Section \ref{sec:partition}.
\begin{enumerate}[leftmargin=1cm]
\item[(Q1)] Is $S$ finitely generated as a right ideal, i.e.\ is $S^r$ finitely generated?
\item[(Q2)] Is the universal right congruence on $S$ finitely generated, i.e.\ is $w_S^r$ finitely generated?
\item[(Q3)] If $w_S^r$ is finitely generated, what is the right diameter $D_r(S)$?
\item[(Q4)] Is $S^l$ finitely generated?
\item[(Q5)] Is $w_S^l$ finitely generated?
\item[(Q6)] If $w_S^l$ is finitely generated, what is the left diameter $D_l(S)$?
\end{enumerate}
Our main results are summarised in Table \ref{table:results}.

\begin{table}[h]
\centering
\begin{tabular}{|c||c|c|c||c|c|c|}
\hline
Semigroup $S$ & $S^r$ f.g.? & $\w_S^r$ f.g.? & $D_r(S)$ & $S^l$ f.g.? & $\w_S^l$ f.g.? & $D_l(S)$\\
\hline\hline
$\B_X$ & Yes & Yes & 1 & Yes & Yes & 1\\
$\PT_X$ & Yes & Yes & 1 & Yes & Yes & 1\\
$\I_X$ & Yes & Yes & 2 & Yes & Yes & 2\\
$\T_X$ & Yes & Yes & 1 & Yes & Yes & 1\\
$\Sym_X$ & Yes & No & n.a. & Yes & No & n.a.\\
$\F_X$ & Yes & Yes & 1 & Yes & No & n.a.\\
$\Inj_X$ & Yes & Yes & 4 & Yes & No & n.a.\\
$\BL_{X,q}, q<|X|$ & Yes & No & n.a. & No & No & n.a.\\
$\BL_X$ & Yes & Yes & 3 & No & No & n.a.\\
$\BL_X^1$ & Yes & Yes & 3 & Yes & No & n.a.\\
$\Sym_X\cup\BL_X$ & Yes & Yes & 4 & Yes & No & n.a.\\
$\Surj_X$ & Yes & No & n.a. & Yes & Yes & 4\\
$\DBL_{X,q}, q<|X|$ & No & No & n.a. & Yes & No & n.a.\\
$\DBL_X$ & No & No & n.a. & Yes & Yes & 2\\
$\DBL_X^1$ & Yes & No & n.a. & Yes & Yes & 3\\
$\Sym_X\cup\DBL_X$ & Yes & No & n.a. & Yes & Yes & 4\\
$\T_X\!\setminus\Inj_X$ & No & No & n.a. & Yes & Yes & 2\\
$\T_X\!\setminus\Surj_X$ & Yes & Yes & 2 & No & No & n.a.\\
$\H_X$ & Yes & Yes & 1 & Yes & No & n.a.\\
$\mathcal{K}_X$ & Yes & No & n.a. & Yes & No & n.a.\\
$\P_X$ & Yes & Yes & 1 & Yes & Yes & 1\\
$\PB_X$ & Yes & Yes & 1 & Yes & Yes & 1\\
\hline
\end{tabular}
\vspace{0.5em}
\caption{Summary of results.}
\label{table:results}
\vspace{-6mm}
\end{table}

For certain transformation semigroups $S,$ we can quickly answer questions (Q1)-(Q6) using known results regarding diagonal acts.

For a semigroup $S,$ the {\em diagonal right $S$-act} is the set $S\times S$ on which $S$ acts on the right via $(a, b)c=(ac, bc).$  It is said to be {\em generated} by a set $U\subseteq S\times S$ if $S\times S=US^1$, and it is {\em finitely generated} or {\em monogenic} if it is generated by a finite set or a singleton, respectively.  Of course, one can dually define the diagonal left $S$-act and its finite generation/monogenicity.

The importance of diagonal acts in relation to the notion of diameter is expressed in the following result.

\begin{prop}\label{prop:diagonal,diameter}\cite[Proposition 3.6]{Gould}
For a non-trivial semigroup $S,$ the diagonal right $S$-act is finitely generated if and only if $S$ has right diameter 1.
\end{prop}

From the substantial body of results on generation of diagonal acts \cite{Gallagher:2006,Gallagher:2005,Gallagher:thesis}, the main findings concerning natural semigroups of transformations and relations are summarised in Table \ref{table:da}.

\begin{table}[h]
\centering
\begin{tabular}{|c||c|c|}
\hline
Semigroup & Diagonal right act & Diagonal left act\\
\hline\hline
$\B_X$ & Monogenic & Monogenic\\
$\PT_X$ & Monogenic & Monogenic\\
$\I_X$ & Not f.g. & Not f.g.\\
$\T_X$ & Monogenic & Monogenic\\
$\Sym_X$ & Not f.g. & Not f.g.\\
$\F_X$ & Monogenic & Not f.g.\\
Infinite subsemigroup of $\Inj_X$ & Not f.g. & Not f.g.\\
Infinite subsemigroup of $\Surj_X$ & Not f.g. & Not f.g.\\
$\T_X\!\setminus\Inj_X$ & Not f.g. & Not f.g.\\
$\T_X\!\setminus\Surj_X$ & Not f.g. & Not f.g.\\
\hline
\end{tabular}
\vspace{0.5em}
\caption{Generation of diagonal acts of certain transformation semigroups.  See \cite[Theorem 7.1 and Lemma 7.2]{Gallagher:2006} for the results on infinite subsemigroups of $\Surj_X$ and $\Inj_X$, and \cite[Table 1.2]{Gallagher:thesis} for the remaining semigroups.}
\label{table:da}
\vspace{-1em}
\end{table}

We immediately deduce from Table \ref{table:da} and Proposition \ref{prop:diagonal,diameter} that $\B_X$, $\PT_X$ and $\T_X$ each have both right diameter 1 and left diameter 1, that $\F_X$ has right diameter 1 but not left diameter 1, and the remaining semigroups appearing in Table \ref{table:da} have neither right diameter 1 nor left diameter 1.  Since $\I_X$ has a zero (the empty map), we deduce, using Corollary \ref{cor:zero} and its left-right dual, that $\I_X$ has right diameter 2 and left diameter 2.

\section{Transformation Semigroups: Right Diameter}
\label{sec:right}

This section naturally divides into three parts, corresponding to questions (Q1), (Q2) and (Q3) of Section \ref{subsec:summary}.  Specifically, we first determine for which of the transformation semigroups $S$ in Table \ref{table:results} we have $S^r$ is {\em not} finitely generated (and hence $\w_S^r$ is not finitely generated).  We then find a number of semigroups $S$ with $S^r$ finitely generated but $\w_S^r$ not finitely generated.  Finally, for each of the remaining semigroups $S,$ we prove that $\w_S^r$ {\em is} finitely generated and determine the right diameter of $S$ (which turns out to be finite).

Now, it is certainly the case that $S^r$ is finitely generated if $S$ is a monoid or a right simple semigroup.  Moreover, it is straightforward to show that $\T_X\!\setminus\!\Surj_X$ is generated as a right ideal of itself by any $\a\in\Inj_X\!\setminus\!\Sym_X$.  So, we are left to consider only $\T_X\!\setminus\!\Inj_X$ and $\DBL_{X,q}$ ($\al\leq q\leq|X|$).  It turns out that these are not finitely generated as right ideals of themselves.  In fact, we prove a stronger result:

\begin{thm}\label{thm:rnotfg}
If $S$ is a finitely transitive subsemigroup of $\T_X\!\setminus\!\Inj_X$, then $S^r$ is not finitely generated.  In particular, the semigroups $\T_X\!\setminus\!\Inj_X$ and $\DBL_{X,q}$ \emph{(}$\al\leq q\leq|X|$\emph{)} are not finitely generated as right ideals of themselves.
\end{thm}

\begin{proof}
Consider any finite subset $U\subseteq S.$  For each $\a\in U,$ choose $(x_\a,y_\a)\in\ker\a$ with $x_\a\neq y_\a$ (such a pair exists because $\a$ is not injective).  Since $S$ is finitely transitive, there exists $\th\in S$ such that $x_\a\th=x_\a$ and $y_\a\th=y_\a$ for all $\a\in U.$  
Thus, for all $\a\in U$ we have $(x_\a,y_\a)\notin\ker\th$ and hence $\theta\notin\a S^1$, so $S\neq US^1$.
Hence, $S^r$ is not finitely generated.
\end{proof}

We now move on to find certain transformation semigroups $S$ for which $S^r$ is finitely generated but $\w_S^r$ is {\em not} finitely generated.  To this end, we first establish a general result regarding generation of the universal right congruence on a subsemigroup of $\T_X$.

Let $S$ be a subsemigroup of $\T_X$.  For a subset $U\subseteq S,$ we define
$$\Sig(U)=\{\a_1\b_1^{-1}\dots\a_k\b_k^{-1} : k\in\N,\,\a_i,\b_i\in U\,(1\leq i\leq k)\}\subseteq\B_X\!.$$
Observe that for any $\th,\phi\in\T_X$, in $\B_X$ we have 
$$\th\phi^{-1}=\{(x,y)\in X\times X : x\th=y\phi\}.$$

\begin{prop}\label{prop:w^r}
Let $S$ be a subsemigroup of $\T_X$, and let $U\subseteq S$ be a generating set for the universal right congruence $\w_S^r$.  Then for any $\th,\phi\in S$ there exists $\s\in\Sig(U)$ with $\s\subseteq\th\phi^{-1}$.
\end{prop}

\begin{proof}
Let $\th,\phi\in S.$  Suppose first that $\th=\phi.$  Since $\w_S^r$ is generated by $U,$ there exists $\a\in U$ such that $\th\in\a S^1.$  Then $\a\a^{-1}\in\Sig(U)$ and
$$\a\a^{-1}=\ker\a\subseteq\ker\th=\th\th^{-1}=\th\phi^{-1}.$$
Now suppose that $\th\neq\phi.$  Then there exists a $U$-sequence
$$\th=\a_1\g_1,\,\b_1\g_1=\a_2\g_2,\,\dots,\,\b_k\g_k=\phi.$$
Let $\s=\a_1\b_1^{-1}\dots\a_k\b_k^{-1}\in\Sig(U).$  We claim that $\s\subseteq\th\phi^{-1}.$  So, let $(x,y)\in\s.$  Then there exist $u_1,v_1,\dots, u_{k-1}, v_{k-1},u_k\in X$ such that
$$(x,u_1)\in\a_1,\,(u_1,v_1)\in\b_1^{-1},\,(v_1,u_2)\in\a_2,\,\dots,\,(v_{k-1},u_k)\in\a_k,\,(u_k,y)\in\b_k^{-1}.$$
Therefore, we have
$$x\a_1=u_1=v_1\b_1,\,v_1\a_2=u_2=v_2\b_2,\,\dots,\,v_{k-1}\a_k=u_k=y\b_k.$$
We then have
$$x\th=x\a_1\g_1=v_1\b_1\g_1
=v_1\a_2\g_2=v_2\b_2\g_2=\cdots=v_{k-1}\a_k\g_k
=y\b_k\g_k=y\phi.$$
Thus $(x,y)\in\th\phi^{-1}$, as required.
\end{proof}

For the next result, recall that the monoid $\mathcal{K}_X$, defined in Section \ref{subsec:trans}, is a submonoid of $\H_X$, and observe that $\mathcal{K}_X\cap\Inj_X=\Inj_X\!\setminus\!\BL_X$.  Note that $\Inj_X\!\setminus\!\BL_X$ contains the symmetric group $\Sym_X$ and the Baer-Levi semigroups $\BL_{X,q}$ where $\al\leq q<|X|$.

\begin{thm}\label{thm:BL,w^rnotfg}
If $S$ is an $\al$-transitive subsemigroup of $\mathcal{K}_X$, then $\w_S^r$ is not finitely generated.  In particular, the universal right congruence is not finitely generated for $\mathcal{K}_X$ or for any $\al$-transitive subsemigroup of $\Inj_X\!\setminus\!\BL_X$ (which includes $\Sym_X$ and $\BL_{X,q}$ where $\al\leq q<|X|$).
\end{thm}

\begin{proof}
First, we claim that for any $\d\in\Sig(S)$ and $Y\subseteq X$ with $|Y|=|X|$ we have $|Y\d|=|X|.$  Indeed, consider $\d=\a_1\b_1^{-1}\dots\a_k\b_k^{-1}\in\Sig(S)$ (where $\a_i,\b_i\in S$) and $Y\subseteq X$ with $|Y|=|X|$.  Define $\d_i=\a_1\b_1^{-1}\dots\a_i\b_i^{-1}$ for $i=0,\dots,k$, interpreting $\d_0=1_X$.  We have $|Y\d_0|=|Y|=|X|.$  Now let $i\in\{1,\dots,k\}$, and assume that $|Y\d_{i-1}|=|X|.$  Then, since $\a_i\in\mathcal{K}_X\subseteq\H_X$, we have $|Y\d_{i-1}\a_i|=|X|.$  It is straightforward to show that $Y\d_{i-1}\a_i\b_i^{-1}\b_i=Y\d_{i-1}\a_i\cap\im\b_i$.  Since 
$$Y\d_{i-1}\a_i=(Y\d_{i-1}\a_i\cap\im\b_i)\cup(Y\d_{i-1}\a_i\!\setminus\!\im\b_i)=(Y\d_{i-1}\a_i\b_i^{-1}\b_i)\cup(Y\d_{i-1}\a_i\!\setminus\!\im\b_i)$$
and $|Y\d_{i-1}\a_i\!\setminus\!\im\b_i|\leq|X\!\setminus\!\im\b_i|<|X|,$ it follows that $|Y\d_{i-1}\a_i\b_i^{-1}\b_i|=|X|.$  Clearly $|Y\d_i|=|Y\d_{i-1}\a_i\b_i^{-1}|\geq|Y\d_{i-1}\a_i\b_i^{-1}\b_i|,$ so $|Y\d_i|=|X|.$  Hence, by finite induction, we have $|Y\d|=|Y\d_k|=|X|.$  This establishes the claim.

Now suppose for a contradiction that $\w_S^r$ is generated by a finite subset $U\subseteq S,$ and let $\Sig=\Sig(U).$  Since $\Sig$ is countable, we may write it as $\Sig=\{\s_i : i\in\N\}$, noting that the $\s_i$ need not be distinct.  Certainly each $\s_i$ belongs to $\Sig(S),$ so it satisfies the condition of the above claim.  Observe that this implies that $\s_i\neq\emptyset.$  We claim that there exist pairs $(x_i,y_i)\in\s_i$ ($i\in\N$) such that $|X\!\setminus\!\{x_i : i\in\N\}|=|X\!\setminus\!\{y_i : i\in\N\}|=|X|.$  This is clear if $X$ is uncountable: for each $i\in\N,$ we can choose any pair $(x_i,y_i)\in\s_i$.  Suppose then that $X$ is countably infinite; we may assume that $X=\N.$  We choose the pairs $(x_i,y_i)$ ($i\in\N$) inductively as follows.  Choose any $(x_1,y_1)\in\s_1.$  For $i\geq2,$ since $|\{x\in X : x\geq x_{i-1}+2\}\s_i|=|X|,$ by the above claim, we can choose $(x_i,y_i)\in\s_i$ with $x_i\geq x_{i-1}+2$ and $y_i\geq y_{i-1}+2.$  Then clearly $X\!\setminus\!\{x_i : i\in\N\}$ and $X\!\setminus\!\{y_i : i\in\N\}$ are infinite, as desired.

We now choose injections $\lam : \{x_i : i\in\N\}\to X$ and $\mu : \{y_i : i\in\N\}\to X$ such that $\im\lam\cap\im\mu=\emptyset.$  Since $S$ is $\al$-transitive, there exist $\th,\phi\in S$ extending $\lam$ and $\mu,$ respectively.  Then $(x_i,y_i)\notin\th\phi^{-1}$ for all $i\in\N.$
Now, by Proposition \ref{prop:w^r} there exists some $i\in\N$ such that $\s_i\subseteq\th\phi^{-1}.$  But then $(x_i,y_i)\in\s\subseteq\th\phi^{-1}$, and we have a contradiction.
\end{proof}

\begin{rem}
The statement and proof of Theorem \ref{thm:BL,w^rnotfg} would still hold if we replaced `finitely generated' with `countably generated'.  This is due to the fact that $\Sig(U)$ is countable for any countable generating set $U$ of $\w_S^r$. 
\end{rem}

It is well known that $\Surj_X$ coincides with the $\el$-class of the identity of $\T_X$.  It follows that every subsemigroup of $\Surj_X$ is $\el^{\ast}$-simple.  Thus, by Proposition \ref{prop:L*-simple}, we have:

\begin{thm}\label{thm:surj,w^rnotfg}
If $S$ is a subsemigroup of $\Surj_X$ such that $\w_S^r$ is finitely generated, then $S$ is countable.  Thus, the universal right congruence on each of the following semigroups is not finitely generated: $\Surj_X$; $\DBL_{X,q}$ where $\al\leq q\leq|X|$; $\DBL_X^1$; $\Sym_X\cup\DBL_X$; and $\Sym_X$.
\end{thm}

The semigroups left to consider in this section are $\H_X$, $\T_X\!\setminus\!\Surj_X$, $\BL_X$, $\BL_X^1$, $\Sym_X\cup\BL_X$ and $\Inj_X$.  For each of these semigroups, we will show that the universal right congruence is finitely generated and determine the right diameter.  

First, we establish certain mappings that will be used repeatedly in the remainder of this section.  These were introduced in \cite[Section 2]{Gallagher:2005} (in the form of binary relations) to prove that $\B_X$, $\PT_X$, $\T_X$ and $\F_X$ each has a monogenic diagonal right act.  We use the `hat' notation to distinguish these mappings from other transformations.

So, let $\ha,\hb\in\T_X$ be two fixed injections such that $\im\ha\cap\im\hb=\emptyset$ and $\im\ha\cup\im\hb=X.$  Note that $\ha,\hb\in\BL_X$.  For each pair $\th,\phi\in\T_X,$ we define a map
$$\hg(\th,\phi) : X\to X, x\mapsto
\begin{cases}
(x\ha^{-1})\th&\text{ if }x\in\im\ha\\
(x\hb^{-1})\phi&\text{ if }x\in\im\hb.
\end{cases}$$
Observe that $\im\hg(\th,\phi)=\im\th\cup\im\phi,$ and $\ha\hg(\th,\phi)=\th$ and $\hb\hg(\th,\phi)=\phi.$  It follows immediately that $\T_X\times\T_X=(\ha,\hb)\T_X.$  Moreover, clearly $\ha,\hb\in\F_X$, and if $\th,\phi\in\F_X$ then $\hg(\th,\phi)\in\F_X$, so we have $\F_X\times\F_X=(\ha,\hb)\F_X.$ 

We fix the maps $\ha,$ $\hb$ and $\hg(\th,\phi)$ ($\th, \phi\in\T_X$) for the remainder of this section.

\begin{defn}
Let $S$ be a subsemigroup of $\T_X$ such that $\ha, \hb\in S,$ let $\th,\phi\in S,$ and let $k\in\N.$  By an $(\ha,\hb,k)${\em-inducing sequence from $\th$ to $\phi$} in $S,$ we mean a sequence $$\th=\psi_1, \psi_2, \dots, \psi_{k+1}=\phi$$ of elements of $S$ where $\hg(\psi_i,\psi_{i+1})\in S$ for each $i\in\{1,\dots,k\}.$
\end{defn}

An $(\ha,\hb,k)$-inducing sequence gives rise to a special kind of $(\ha,\hb)$-sequence of length $k$, and vice versa:

\begin{lem}\label{lem:inducing,right}
Let $S$ be a subsemigroup of $\T_X$ such that $\ha, \hb\in S,$ let $\th,\phi\in S,$ and let $k\in\N.$  Then the following statements are equivalent.
\begin{enumerate}
\item There exists an $(\ha,\hb,k)$-inducing sequence
$$\th=\psi_1, \psi_2, \dots, \psi_{k+1}=\phi$$ 
from $\th$ to $\phi$ in $S.$
\item There exists an $(\ha,\hb)$-sequence
$$\th=\ha\g_1,\, \hb\g_1=\ha\g_2,\, \dots,\, \hb\g_{k-1}=\ha\g_k,\, \hb\g_k=\phi$$
from $\th$ to $\phi$ of length $k$ in $S.$
\end{enumerate}
\end{lem}

\begin{proof}
(1)$\Rightarrow$(2).  By the definition of an $(\ha,\hb,k)$-inducing sequence, we have $\hg(\psi_i,\psi_{i+1})\in S$ for each $i\in\{1,\dots,k\}.$  Letting $\g_i=\hg(\psi_i,\psi_{i+1}),$ we have $\ha\g_i=\psi_i$ and $\hb\g_i=\psi_{i+1}.$  Hence, there is an $(\ha,\hb)$-sequence
$$\th=\ha\g_1,\, \hb\g_1=\ha\g_2,\, \dots,\, \hb\g_{k-1}=\ha\g_k,\, \hb\g_k=\phi$$
in $S.$

(2)$\Rightarrow$(1).  By the definition of an $(\ha,\hb)$-sequence in $S,$ we have $\g_i\in S$ for $1\leq i\leq k.$  Let $\psi_i=\ha\g_i$ for each $i\in\{1,\dots,k\},$ and let $\psi_{k+1}=\phi.$  Then $\hg(\psi_i,\psi_{i+1})=\hg(\ha\g_i,\hb\g_i)$ for each $i\in\{1,\dots,k\}.$  Consider any $x\in X.$  Then $x\in\im\ha$ or $x\in\im\hb.$  If $x\in\im\ha$ then 
$$x\hg(\psi_i,\psi_{i+1})=x\hg(\ha\g_i,\hb\g_i)=(x\ha^{-1})\ha\g_i=x\g_i.$$
Similarly, if $x\in\im\hb$ then $x\hg(\psi_i,\psi_{i+1})=x\g_i$.  Thus $\hg(\psi_i,\psi_{i+1})=\g_i\in S.$  Hence, there is an $(\ha,\hb,k)$-inducing sequence 
$$\th=\psi_1, \psi_2, \dots, \psi_{k+1}=\phi$$ 
in $S.$
\end{proof}

Lemma \ref{lem:inducing,right} yields the following result.

\begin{prop}\label{prop:rightdiameter}
Let $S$ be a subsemigroup of $\T_X$ such that:
\begin{enumerate}
\item $\ha,\hb\in S$; 
\item there exists $n\in\N$ such that for any pair $\th, \phi\in S$ there is an $(\ha,\hb,k)$-inducing sequence from $\th$ to $\phi$ in $S$ for some $k\leq n.$
\end{enumerate} 
Then $\w_S^r$ is generated by the pair $(\ha,\hb)$ and $D_r(S)\leq n.$  Furthermore, if $n=1$ (so that $\hg(\th,\phi)\in S$ for any $\th,\phi\in S$), then the diagonal right $S$-act is generated by $(\ha,\hb)$ (and is hence monogenic).
\end{prop}

Using Proposition \ref{prop:rightdiameter}, we show that the diagonal right act of $\H_X$ is monogenic.

\begin{thm}\label{thm:H}
The diagonal right act of $\H_X$ is generated by $(\ha,\hb),$ and consequently $\H_X$ has right diameter 1.
\end{thm}

\begin{proof}
Clearly $\ha,\hb\in\H_X$.  Let $\th,\phi\in\H_X$, and write $\g=\hg(\th,\phi).$  By Proposition \ref{prop:rightdiameter}, it suffices to prove that $\g\in\H_X$.  So, let $Y\subseteq X$ with $|Y|=|X|.$  We have
$$Y\g=(Y\cap\im\ha)\ha^{-1}\th\cup(Y\cap\im\hb)\hb^{-1}\phi.$$
Since $Y=(Y\cap\im\ha)\cup(Y\cap\im\hb)$, and $\ha,\hb$ are bijections, it follows that at least one of $(Y\cap\im\ha)\ha^{-1}$ and $(Y\cap\im\hb)\hb^{-1}$ has cardinality $|X|.$  Since $\th,\phi\in\H_X,$ we conclude that at least one of $(Y\cap\im\ha)\ha^{-1}\th$ and $(Y\cap\im\hb)\hb^{-1}\phi$ has cardinality $|X|,$ and hence $|Y\g|=|X|.$  Thus $\g\in\H_X,$ as required.  
\end{proof}

By the proof of \cite[Corollary 1]{Higgins}, any semigroup can be embedded in some $\H_X$.  This fact, together with Theorem \ref{thm:H}, yields:

\begin{cor}
Any semigroup can be embedded in a bisimple monoid whose diagonal right act is monogenic.
\end{cor}

We now move on to consider $\T_X\!\setminus\Surj_X$.

\begin{thm}
The semigroup $\T_X\!\setminus\Surj_X$ has right diameter 2.
\end{thm}

\begin{proof}
Let $S=\T_X\!\setminus\Surj_X$.  By Table \ref{table:da} and Proposition \ref{prop:diagonal,diameter}, $S$ does not have right diameter 1.  Using Proposition \ref{prop:rightdiameter}, we show that $\w_S^r=\langle(\ha,\hb)\rangle$ with $D_r(S)\leq 2,$ and hence $D_r(S)=2.$  

Clearly $\ha,\hb\in S.$  Consider any $\th,\phi\in S.$  Since $\th$ and $\phi$ are not surjective, we can choose $y\in X$ such that $\im\th\cup\{y\}\neq X$ and $\im\phi\cup\{y\}\neq X.$ Letting $c_y$ denote the constant map with image $y,$ we have
$\im\hg(\th,c_y)=\im\th\cup\{y\}\neq X$ and $\im\hg(c_y,\phi)=\im\phi\cup\{y\}\neq X,$ so that $\hg(\th,c_y),\hg(c_y,\phi)\in S.$  Thus, we have an $(\ha,\hb,2)$-inducing sequence $\th,\,c_y,\,\phi,$ as required.
\end{proof}

We now turn our attention to $\BL_X$, $\BL_X^1$, $\Sym_X\cup\BL_X$ and $\Inj_X$.  In fact, we will obtain results concerning a larger class of subsemigroups of $\Inj_X.$  We begin with the following technical lemma.

\begin{lem}\label{lem:BL}
For any $\th,\phi\in\BL_X$ such that $|X\!\setminus\!(\im\th\cup\im\phi)|=|X|,$ there exists 
an $(\ha,\hb)$-sequence from $\th$ to $\phi$ of length 2 (in $\BL_X$).
\end{lem}

\begin{proof}
Let $S=\BL_X$.  By Lemma \ref{lem:inducing,right}, it suffices to show that there exists an $(\ha,\hb,2)$-inducing sequence from $\th$ to $\phi$; that is, there exists $\lam\in S$ such that $\hg(\th,\lam),\hg(\lam,\phi)\in S.$

Let $Z=\im\th\cup\im\phi.$  Then $|X\!\setminus\!Z|=|X|$ by assumption.  Let $\lam : X\to X\!\setminus\!Z$ be an injection such that $|X\setminus(Z\cup\im\lambda)|=|X|.$  Clearly $\lam\in S.$  Let $\g_1=\hg(\th,\lam)$ and $\g_2=\hg(\lam,\phi).$  It is straightforward to show that $\g_1,\g_2\in\Inj_X$.  Moreover, we have
$$\im\g_1=\im\th\cup\im\lam
\subseteq Z\cup\im\lam,$$
so $|X\!\setminus\!\im\g_1|\geq|X\!\setminus\!(Z\cup\im\lam)|=|X|.$  Thus $\g_1\in S,$ and similarly $\g_2\in S,$ as desired.
\end{proof}

The following result provides several equivalent characterisations for an $|X|$-transitive subsemigroup of $\Inj_X$ to have right diameter 3 or 4. 

\begin{prop}\label{prop:Inj,BL}
For a subsemigroup $S$ of $\Inj_X$, the following are equivalent:
\begin{enumerate}
\item $S$ is $|X|$-transitive, $\w_S^r$ is finitely generated and $D_r(S)\in\{3,4\}$;
\item $S$ is $|X|$-transitive and $\w_S^r$ is finitely generated;
\item $S$ is $|X|$-transitive, $S^r$ is finitely generated and $S\cap\BL_X\neq\emptyset$;
\item $S^r$ is finitely generated and $S$ contains $\BL_X$.
\end{enumerate}
\end{prop}

\begin{proof}
(1)$\Rightarrow$(2) is trivial.

(2)$\Rightarrow$(3).  By Lemma \ref{lem:w}, $S^r$ is finitely generated, and it follows from Theorem \ref{thm:BL,w^rnotfg} that $S\cap\BL_X\neq\emptyset$.

(3)$\Rightarrow$(4).  Fix any $\a\in S\cap\BL_X$, and consider an arbitrary $\b\in\BL_X$.  Then $\a^{-1}\b\in\I_X$.  We have $\dom\a^{-1}\b=\im\a$ and $\im\a^{-1}\b=\im\b,$ so that 
$$|X\!\setminus\!\dom\a^{-1}\b|=|X\!\setminus\!\im\a|=|X|\quad\text{ and }\quad|X\!\setminus\!\im\a^{-1}\b|=|X\!\setminus\!\im\b|=|X|.$$
Since $S$ is $|X|$-transitive, there exists $\g\in S$ extending $\a^{-1}\b$, i.e.\ $\g|_{\im\a}=\a^{-1}\b.$  Therefore, for each $x\in X$ we have $(x\a)\g=(x\a)\a^{-1}\b=x\b,$ so that $\b=\a\g\in S.$  Thus $\BL_X\subseteq S.$

(4)$\Rightarrow$(1).  Since $S$ contains $\BL_X$, which is $|X|$-transitive, $S$ is $|X|$-transitive.

We now prove that $\w_S^r$ is finitely generated with $D_r(S)\leq4.$  By assumption, there exists a finite subset $V\subset S$ such that $S^r=VS^1$.  Let $K=\BL_X$, and recall that $\ha,\hb\in K.$  Letting $U=V\cup\{\ha,\hb\},$ we shall prove that $\w_S^r=\l U\r$ with $D_r(S,U)\leq 4.$ 

So, let $\th,\phi\in S.$  We claim that there exist $\th',\phi'\in K$ such that the pairs $(\th,\th'),(\phi',\phi)\in\w_S^r$ are each obtained by a single application of a pair from $U\times U,$ and $|X\!\setminus\!(\im\th'\cup\im\phi')|=|X|.$

Indeed, we have $\th=\g\s$ and $\phi=\d\tau$ for some $\g,\d\in V$ and $\s,\tau\in S^1.$  Let $\th'=\ha\s.$  Then clearly $(\th,\th')$ is obtained by a single application of the pair $(\g,\ha)\in U\times U.$  Now, since $\im\ha\cap\im\hb=\emptyset$ and $\tau$ is injective, we have $\im\ha\tau\cap\im\hb\tau=\emptyset.$
Thus,
$$(X\!\setminus\!\im\th')\cap\im\ha\tau\subseteq X\!\setminus\!(\im\th'\cup\im\hb\tau).$$
Therefore, we have
\begin{align*}
X\!\setminus\!\im\th'&=X\!\setminus\!(\im\th'\cup\im\ha\tau)\cup\big((X\!\setminus\!\im\th')\cap\im\ha\tau\big)\\
&\subseteq X\!\setminus\!(\im\th'\cup\im\ha\tau)\cup X\!\setminus\!(\im\th'\cup\im\hb\tau).
\end{align*}
Since $|X\!\setminus\!\im\th'|=|X|$ (as $\th'\in K$), it follows that either
$$|X\!\setminus\!(\im\th'\cup\im\ha\tau)|=|X|\quad\text{or}\quad|X\!\setminus\!(\im\th'\cup\im\hb\tau)|=|X|.$$
If $|X\!\setminus\!(\im\th'\cup\im\ha\tau)|=|X|,$ we set $\phi'=\ha\tau$; otherwise, we set $\phi'=\hb\tau.$  Then $(\phi',\phi)$ is obtained by a single application of either $(\ha,\d)$ or $(\hb,\d),$ and $|X\!\setminus\!(\im\th'\cup\im\phi')|=|X|.$  This completes the proof of the claim.

Now, by Lemma \ref{lem:BL}, there exists an $(\ha,\hb)$-sequence from $\th'$ to $\phi'$ of length 2 in $K\subseteq S.$  It follows that there is a $U$-sequence from $\th$ to $\phi$ of length 4.  Hence, $D_r(S,U)\leq4,$ as desired.\\
~\\
Now, to prove the lower bound of 3 for $D_r(S),$ suppose for a contradiction that $D_r(S,U)\leq2$ for some finite set $U\subseteq S\times S.$  

We say that a pair $(\g,\d)\in S\times S$ is {\em disjoint} if $\im\g\cap\im\d=\emptyset,$ and {\em intersecting} otherwise.  We may assume that $U$ contains an intersecting pair, for otherwise we can add such a pair to $U.$  Let $\{(\g_i,\d_i) : 1\leq i\leq n\}$ be the set of intersecting pairs in $U.$  For each $i\in\{1,\dots,n\}$ choose $x_i\in X$ such that $x_i\g_i\in\im\d_i$, and let $y_i=x_i\g_i\d_i^{-1}$.  Now let 
$$Q=\big\{(j,k) : j,k\in\{1,\dots,n\},\, y_j\in\im\d_k\g_k^{-1}\big\}.$$
Choose 
$$w_1,\dots,w_n\in X\!\setminus\!\{y_i,\, y_j\g_k\d_k^{-1} : 1\leq i\leq n,\, (j,k)\in Q\}$$ such that $w_i=w_j$ if and only if $x_i=x_j$.  Fix any $\th\in\BL_X~(\subseteq S),$ and note that $|\!\im\th|=|X\!\setminus\!\im\th|=|X|.$  Choose
$$z_1,\dots,z_n\in\im\th\!\setminus\!\{x_i\th : 1\leq i\leq n\}$$
such that $z_i=z_j$ if and only if $y_i=y_j$.  Note that the sets 
$$A=\{w_i, y_i : 1\leq i\leq n\}\quad\text{and}\quad B=\{x_i\th, z_i : 1\leq i\leq n\}$$ are finite and have the same cardinality.  Choose a bijection 
$$\lam : X\!\setminus\!(\im\th\cup A)\to X\!\setminus\!(\im\th\cup B),$$
and extend $\lam$ to a bijection 
$$\lam' : (X\!\setminus\!\im\th)\cup A\to(X\!\setminus\!\im\th)\cup B$$
by setting $w_i\lam'=x_i\th$ and $y_i\lam'=z_i$ ($1\leq i\leq n$).  We have
$$X\!\setminus\!\dom\lam'=\im\th\!\setminus\!A\quad\text{and}\quad X\!\setminus\!\im\lam'=\im\th\!\setminus\!B.$$
Since $|\im\th|=|X|$ and $A$ and $B$ are finite, we have
$$|X\!\setminus\!\dom\lam'|=|X\!\setminus\!\im\lam'|=|X|.$$
Since $S$ is $|X|$-transitive, there exists some $\phi\in S$ extending $\lam'.$  Note that $\im\th\cup\im\phi=X.$
Since $D_r(S,U)\leq2,$ there exists a $U$-sequence
$$\th=\g\s,\; \d\s=\g'\s',\; \d'\s'=\phi.$$
First suppose that $(\g,\d)$ and $(\g',\d')$ are disjoint pairs.  If the pair $(\th,\d\s)$ were intersecting, then there would exist $x,y\in X$ such that $x\th=(x\g)\s=y\d\s,$ but then $x\g=y\d$ since $\s$ is injective, contradicting that $(\g,\d)$ is disjoint.  Thus $(\th,\d\s)$ is disjoint, and similarly $(\d\s,\phi)=(\g'\s',\phi)$ is disjoint.  Thus, we have $$\im\d\s\cap(\im\th\cup\im\phi)=\emptyset.$$ 
But $\im\th\cup\im\phi=X,$ so we have a contradiction. 
We conclude that at least one of $(\g,\d)$ and $(\g',\d')$ is intersecting.  

Suppose first that $(\g,\d)$ is intersecting, so that $(\g,\d)=(\g_j,\d_j)$ for some $j\in\{1,\dots,n\}.$  We then have
$$y_j\g'\s'=y_j\d_j\s=x_j\g_j\s=x_j\th=w_j\phi=w_j\d'\s'.$$
Since $\s'$ is injective, we have $y_j\g'=w_j\d'.$  Thus $(\g',\d')$ is intersecting, so that $(\g',\d')=(\g_k,\d_k)$ for some $k\in\{1,\dots,n\}.$  Hence $y_j\g_k=w_j\d_k.$  But then $(j,k)\in Q$ and $w_j=y_j\g_k\d_k^{-1} $, contradicting the choice of $w_j$.  

Now suppose that $(\g',\d')$ is intersecting, so that $(\g',\d')=(\g_l,\d_l)$ for some $l\in\{1,\dots,n\}.$  Since $z_l\in\im\th,$ there exists some $x\in X$ such that $z_l=x\th.$  Thus, we have
$$x\g\s=x\th=z_l=y_l\phi=y_l\d_l\s'=x_l\g_l\s'=x_l\d\s.$$
Since $\s$ is injective, it follows that $x\g=x_l\d.$  But it has already been established that $(\g,\d)$ is not intersecting, so we have a contradiction.  Thus $D_r(S)\geq3.$  This completes the proof of (4)$\implies$(1) and hence of the proposition.
\end{proof}

It follows immediately from Proposition \ref{prop:Inj,BL} that each of $\BL_X$, $\BL_X^1$, $\Sym_X\cup\BL_X$ and $\Inj_X$ has right diameter either 3 or 4.  We shall prove that the former two have right diameter 3 and the latter two have right diameter 4.

Since $\Inj_X\!\setminus\!\BL_X$ is a submonoid of $\Inj_X$, for any subsemigroup $S$ of $\Inj_X$ we have that $S\!\setminus\!\BL_X$ is a (possibly empty) subsemigroup of $S.$  If $S\!\setminus\!\BL_X$ is finite and non-empty, then it is a subgroup of $\Sym_X$; this follows from the fact that $\Inj_X\!\setminus\!\Sym_X$ contains no idempotents.  

\begin{thm}\label{thm:BL} 
If $S$ is a semigroup such that $\BL_X\leq S\leq\Inj_X$ and $S\!\setminus\!\BL_X$ is finite, then $D_r(S)=3.$
In particular, $\BL_X$ and $\BL_X^1$ have right diameter 3.
\end{thm}

\begin{proof} 
By Proposition \ref{prop:Inj,BL}, it suffices to prove that $D_r(S)\leq 3.$  

Let $U=(S\!\setminus\!\BL_X)\cup\{\ha,\hb\}.$  Certainly $U$ is finite since $S\!\setminus\!\BL_X$ is finite.  We shall prove that $\w_S^r=\l U\r$ and $D_r(S,U)\leq3.$

So, consider $\th,\phi\in S.$  If $\th,\phi\in S\!\setminus\!\BL_X$, then $\th,\phi\in U,$ so clearly there is a $U$-sequence from $\th$ to $\phi$ of length 1.   Assume then that $\th\in\BL_X$, and suppose first that $\phi\in S\!\setminus\!\BL_X$.  Since $X=\im\ha\cup\im\hb$ and $|X\!\setminus\!\im\th|=|X|,$ it follows that either $|X\!\setminus\!(\im\th\cup\im\ha)|=|X|$ or $|X\!\setminus\!(\im\th\cup\im\hb)|=|X|.$  Assume without loss of generality that $|X\!\setminus\!(\im\th\cup\im\ha)|=|X|.$  Then, by Lemma \ref{lem:BL}, there exists an $(\ha,\hb)$-sequence from $\th$ to $\ha$ of length 2.  Since $\ha,\phi\in U$, we conclude that there is a $U$-sequence from $\th$ to $\phi$ of length 3.

Finally, suppose that $\phi\in\BL_X$.
If $|X\!\setminus\!(\im\th\cup\im\phi)|=|X|,$ then by Lemma \ref{lem:BL} there exists an $(\ha,\hb)$-sequence of length 2 from $\th$ to $\phi.$  Suppose then that $|X\!\setminus\!(\im\th\cup\im\phi)|<|X|.$  Let $Y=\im\phi\!\setminus\!\im\th.$  Since $$X\!\setminus\!\im\th=\bigl(X\!\setminus\!(\im\th\cup\im\phi)\bigr)\cup Y$$
and $|X\!\setminus\!\im\th|=|X|,$ it follows that $|Y|=|X|.$  Let $\lam : X\to Y$ be an injection such that $|Y\!\setminus\!\im\lam|=|X|,$ and let $\g=\hg(\th,\lam).$  Clearly $\g$ is an injection (and hence $|\im\g|=|X|$).  Also, we have
$$Y\!\setminus\!\im\lam\subseteq X\!\setminus\!(\im\th\cup\im\lam)=X\!\setminus\!\im\g.$$
Since $|Y\!\setminus\!\im\lam|=|X|,$ it follows that $|X\!\setminus\!\im\g|=|X|.$  Thus $\g\in S.$  Recall that $\th=\ha\g$ and $\lam=\hb\g.$
Now, since $\im\lam\subseteq Y\subseteq\im\phi,$ we have that 
$$|X\!\setminus\!(\im\lam\cup\im\phi)|=|X\!\setminus\im\phi|=|X|.$$
Therefore, by Lemma \ref{lem:BL}, there exists an $(\ha,\hb)$-sequence from $\lam$ to $\phi$ of length 2.  Hence, we have an $(\ha,\hb)$-sequence from $\th$ to $\phi$ of length 3.  This completes the proof.
\end{proof}

Next, we show that a subsemigroup $S$ of $\Inj_X$ such that $S\!\setminus\!\BL_X$ is finitely transitive cannot have right diameter strictly less than 4.

\begin{prop}\label{prop:S,Inj}
Let $S$ be a subsemigroup of $\Inj_X$ such that $S\setminus\BL_X$ is finitely transitive.  If $\w_S^r$ is finitely generated, then $D_r(S)\geq4$.
\end{prop}

\begin{proof}
Suppose for a contradiction that $D_r(S,U)\leq3$ for some finite set $U\subseteq S.$  Let $P$ denote the (finite) collection of all tuples $(\a_1,\b_1,\a_2,\b_2,\a_3,\b_3)\in U^6$ where $\a_1,\b_3\in S\setminus\BL_X$.  Observe that for any $\a\in S\!\setminus\!\BL_X$ and $\b\in S,$ since $|X\!\setminus\!\im\a|<|X|$ we have $|\!\im\a\cap\im\b|=|X|$, or, equivalently, in $\I_X$ we have $|\!\im\a\b^{-1}|=|\!\im\b\a^{-1}|=|X|.$  Therefore, we may choose a set of distinct elements $\{x_p,y_p : p\in P\}$ such that, for each $p=(\a_1,\dots,\b_3)\in P,$
in $\I_X$ we have $x_p\in\im\b_1\a_1^{-1}$, $y_p\in\im\a_3\b_3^{-1}$ and $x_p\a_1\b_1^{-1}\a_2\neq y_p\b_3\a_3^{-1}\b_2$.

Since $S\!\setminus\!\BL_X$ is finitely transitive, there exist $\th,\phi\in S\!\setminus\!\BL_X$ such that $x_p\th=x_p$ and $y_p\phi=x_p$ for all $p\in P.$  As $D_r(S,U)\leq3,$ there exists a $U$-sequence
$$\th=\a_1\g_1,\,\b_1\g_1=\a_2\g_2,\,\b_2\g_2=\a_3\g_3,\,\b_3\g_3=\phi.$$
Since $\BL_X$ is an ideal of $S,$ it follows that $\a_1,\b_3\in S\!\setminus\!\BL_X$, so that $(\a_1,\dots,\b_3)\in P.$  Letting $(\a_1,\dots,\b_3)=p,$ we have 
\begin{align*}
x_p\a_1\b_1^{-1}\a_2\g_2=x_p\a_1\b_1^{-1}\b_1\g_1=x_p\a_1\g_1=x_p\th=x_p=y_p\phi=y_p\b_3\g_3
&=y_p\b_3\a_3^{-1}\a_3\g_3\\
&=y_p\b_3\a_3^{-1}\b_2\g_2.
\end{align*}
But then, since $\g_2$ is injective, we have $x_p\a_1\b_1^{-1}\a_2=y_p\b_3\a_3^{-1}\b_2$, contradicting the choice of $x_p$ and $y_p$.  Thus $D_r(S)\geq4.$
\end{proof}

If $S$ is a subsemigroup of $\Inj_X$ such that $S\!\setminus\!\BL_X$ is $|X|$-transitive, then certainly $S$ is $|X|$-transitive and $S\!\setminus\!\BL_X$ is finitely transitive.  Thus, by Propositions \ref{prop:Inj,BL} and \ref{prop:S,Inj}, we have:

\begin{thm}\label{thm:S,Inj}
For a subsemigroup $S$ of $\Inj_X$ such that $S\!\setminus\!\BL_X$ is $|X|$-transitive, the following are equivalent:
\begin{enumerate}
\item $\w_S^r$ is finitely generated and $D_r(S)=4$;
\item $\w_S^r$ is finitely generated;
\item $S^r$ is finitely generated and $S\cap\BL_X\neq\emptyset$;
\item $S^r$ is finitely generated and $S$ contains $\BL_X$.
\end{enumerate}
\end{thm}

If $S$ is a subsemigroup of $\Inj_X$ containing $\Sym_X$, then certainly $S^r$ is finitely generated (since $S$ is a monoid) and $S\!\setminus\!\BL_X$ is $|X|$-transitive (since it contains $\Sym_X$, which is $|X|$-transitive).  Thus, we deduce:

\begin{thm}\label{thm:Sym,Inj}
For a monoid $S$ such that $\Sym_X\leq S\leq\Inj_X$, the following are equivalent:
\begin{enumerate}
\item $\w_S^r$ is finitely generated and $D_r(S)=4$;
\item $\w_S^r$ is finitely generated;
\item $S\cap\BL_X\neq\emptyset$;
\item $S$ contains $\BL_X$.
\end{enumerate}
Consequently, the monoids $\Sym_X\cup\BL_X$ and $\Inj_X$ have right diameter 4.
\end{thm}

\begin{rem}
For any non-empty set $I$ of infinite cardinals $q\leq|X|,$ the set $S=\bigcup_{q\in I}\BL_{X,q}$ is an $|X|$-transitive subsemigroup of $\Inj_X$.  Moreover, we have $S^r=\a S^1$ for any $\a\in\BL_{X,q_0}$, where $q_0$ is the smallest cardinal in $I.$  If $I$ contains $|X|$ and at least one other cardinal $q<|X|,$  then, by Theorem \ref{thm:S,Inj}, the universal right congruence $\w_S^r$ is finitely generated and $D_r(S)=4.$
\end{rem}

\section{Transformation Semigroups: Left Diameter}
\label{sec:left}

This section has a parallel structure to Section \ref{sec:right}; that is, it naturally splits into three parts, correponding to questions (Q4), (Q5) and (Q6) of Section \ref{subsec:summary}.

So, we begin by considering which of the transformation semigroups $S$ appearing in Table \ref{table:results} are finitely generated as left ideals.  Of course, this holds if $S$ is a monoid or left simple.  Also, it is fairly straightforward to show that $\T_X\!\setminus\!\Inj_X$ is generated as a left ideal of itself by any $\a\in\Surj_X\!\setminus\!\Sym_X$ (in fact, we shall see that $\T_X\!\setminus\!\Inj_X$ has left diameter 2).  The remaining semigroups ($\T_X\!\setminus\!\Surj_X$ and $\BL_{X,q}$, $\al\leq q\leq|X|$) are dealt with by the following result.

\begin{thm}\label{thm:lnotfg}
If $S$ is a finitely transitive subsemigroup of $\T_X\!\setminus\!\Surj_X$, then $S^l$ is not finitely generated.  In particular, the semigroups $\T_X\!\setminus\!\Surj_X$ and $\BL_{X,q}$ \emph{(}$\al\leq q\leq|X|$\emph{)} are not finitely generated as left ideals of themselves. 
\end{thm}

\begin{proof}
Consider any finite subset $U\subseteq S.$  For each $\a\in U$ choose $x_\a\in X\!\setminus\!\im\a$ (such an element exists because $\a$ is not surjective).  Since $S$ is finitely transitive, there exists $\th\in S$ such that $x_\a\th=x_\a$ for each $\a\in U.$  Then $\im\th\not\subseteq\im\a$ for all $\a\in U,$ and hence $\th\notin S^1\a$ for any $\a\in U,$ so $S\neq S^1U.$  Hence, $S^l$ is not finitely generated.
\end{proof}

We now consider which of the remaining transformation semigroups $S$ from Table \ref{table:results} have $\w_S^l$ not finitely generated.  First, we establish an analogue of Proposition \ref{prop:w^r}, followed by a technical lemma.

Let $S$ be a subsemigroup of $\T_X$.  For a subset $U\subseteq S,$ we define
$$\Sig'(U)=\{\a_1^{-1}\b_1\dots\a_k^{-1}\b_k : k\in\N,\,\a_i,\b_i\in U\,(1\leq i\leq k)\}\subseteq\B_X\!.$$
Observe that for any $\th,\phi\in\T_X$, in $\B_X$ we have 
$$\th^{-1}\phi=\{(x\th,x\phi) : x\in X\}.$$

\begin{prop}\label{prop:w^l}
Let $S$ be a subsemigroup of $\T_X$, and let $U\subseteq S$ be a generating set for the universal left congruence $\w_S^l$.  Then for any $\th,\phi\in S$ there exists $\s\in\Sig'(U)$ with $\th^{-1}\phi\subseteq\s$.
\end{prop}

\begin{proof}
Let $\th,\phi\in S.$  Suppose first that $\th=\phi.$  Since $\w_S^l$ is generated by $U,$ there exists $\a\in U$ such that $\th\in S^1\a.$  Since $\im\th\subseteq\im\a,$ it follows that $\th^{-1}\phi=\th^{-1}\th\subseteq\a^{-1}\a\in\Sig'(U).$

Now suppose that $\th\neq\phi.$  Then there exists a $U$-sequence
$$\th=\g_1\a_1,\,\g_1\b_1=\g_2\a_2,\,\dots,\,\g_k\b_k=\phi.$$
Let $\s=\a_1^{-1}\b_1\dots\a_k^{-1}\b_k\in\Sig'(U).$  We claim that $\th^{-1}\phi\subseteq\s.$  So, let $(x,y)\in\th^{-1}\phi$.  Then there exists $z\in X$ such that $x=z\th$ and $y=z\phi.$
Then $x=(z\g_1)\a_1,$ so that $(x,z\g_1)\in\a_1^{-1}.$  Therefore, we have
$$(x,z\g_2\a_2)=(x,z\g_1\b_1)\in\a_1^{-1}\b_1.$$  
It follows that $(x,z\g_2)\in\a_1^{-1}\b_1\a_2^{-1},$ which in turn implies that $(x,z\g_2\b_2)\in\a_1^{-1}\b_1\a_2^{-1}\b_2.$  Continuing in this way, we obtain
$$(x,y)=(x,z\phi)=(x,z\g_k\b_k)\in\a_1^{-1}\b_1\dots\a_k^{-1}\b_k=\s,$$ 
as required. 
\end{proof}

\begin{lem}\label{lem:w^lnotfg}
Let $S$ be an $\al$-transitive subsemigroup of $\T_X$ satisfying the following condition: for any finite subset $U\subset S$ there are infinitely many $x\in X$ such that for each $\s\in\Sig'(U)$ the set $X\!\setminus\!x\s$ is infinite.  Then $\w_S^l$ is not finitely generated.
\end{lem}

\begin{proof}
Suppose for a contradiction that $\w_S^l$ is generated by a finite subset $U\subset S,$ and let $\Sig=\Sig'(U).$   Let $X'$ be the (infinite) set of all $x\in X$ such that for each $\s\in\Sig$ the set $X\!\setminus\!x\s$ is infinite.
Choose a set of distinct elements $\{x_{\s} : \s\in\Sig\}\subseteq X'$ such that $|X\!\setminus\!\{x_{\s} : \s\in\Sig\}|=|X|.$
Since the set $X\!\setminus\!x_{\s}\s$ is infinite for each $\s\in\Sig,$ 
we may choose a set of distinct elements $\{y_{\s} : \s\in\Sig\}$ such that $|X\!\setminus\!\{y_{\s} : \s\in\Sig\}|=|X|$ and $(x_{\s},y_{\s})\notin\s$ for each $\s\in\Sig.$
Since $\Sig$ is countable and $S$ is $\al$-transitive, there exist $\th,\phi\in S$ such that $x_{\s}\th=x_{\s}$ and $x_{\s}\phi=y_{\s}$ for all $\s\in\Sig.$  Then $(x_{\s},y_{\s})\in\th^{-1}\phi$ for all $\s\in\Sig.$  Now, by Proposition \ref{prop:w^l}, there exists $\s\in\Sig$ with $\th^{-1}\phi\subseteq\s$.  But then $(x_{\s},y_{\s})\in\s,$ and we have a contradiction.
\end{proof}


We can now show that all the subsemigroups of $\H_X$ appearing in Table \ref{table:results} do not have a finitely generated universal left congruence.

\begin{thm}\label{thm:w^lnotfg}
If $S$ is an $\al$-transitive subsemigroup of $\H_X$, then $\w_S^l$ is not finitely generated.
In particular, the universal left congruence on each of the following semigroups is not finitely generated: $\Sym_X$; $\F_X$; $\Inj_X$; $\BL_{X,p}$ where $\al\leq p\leq|X|$; $\BL_X^1$; $\Sym_X\cup\BL_X$; $\DBL_{X,q}$ where $\al\leq q<|X|$; $\H_X$; and $\mathcal{K}_X$.
\end{thm}

\begin{proof}
We claim that $S$ satisfies the condition of Lemma \ref{lem:w^lnotfg}, and hence $\w_S^l$ is not finitely generated.  Indeed, consider any $x\in X,$ $U\subseteq S$ and $\s=\a_1^{-1}\b_1\dots\a_k^{-1}\b_k\in\Sig(U)$ (where $\a_i,\b_i\in U$).  Define $\s_i=\a_1^{-1}\b_1\dots\a_i^{-1}\b_i$ for $i=0,\dots,k$ (interpreting $\s_0=1_X$).  We have $|x\s_0|=|\{x\}|<|X|.$  Now let $i\in\{1,\dots,k\}$, and assume that $|x\s_{i-1}|<|X|.$
We have $x\s_{i-1}\a_i^{-1}\a_i=x\s_{i-1}\cap\im\a_i$, so $|x\s_{i-1}\a_i^{-1}\a_i|\leq|x\s_{i-1}|<|X|.$  Since $\a_i\in\H_X$, it follows that $|x\s_{i-1}\a_i^{-1}|<|X|.$  Therefore, 
$$|x\s_i|=|x\s_{i-1}\a_i^{-1}\b_i|\leq|x\s_{i-1}\a_i^{-1}|<|X|.$$
Hence, by induction, we have $|x\s|=|x\s_k|<|X|.$
Thus $X\!\setminus\!x\s$ is infinite, as required.
\end{proof}

\vspace{0.5em}
\begin{rem}~
\begin{enumerate}[leftmargin=*]
\item The statements and proofs of Lemma \ref{lem:w^lnotfg} and Theorem \ref{thm:w^lnotfg} would still hold if we replaced `finitely generated' with `countably generated'.
\item The monoid $\Inj_X$ coincides with the $\ar$-class of the identity of $\T_X,$ so every subsemigroup of $\Inj_X$ is $\ar^{\ast}$-simple.  Thus, by the left-right dual of Proposition \ref{prop:L*-simple}, the universal left congruence $\w_S^l$ on any uncountable subsemigroup $S$ of $\Inj_X$ is not finitely generated.
\end{enumerate}
\end{rem}

The remaining semigroups to consider are $\T_X\!\setminus\!\Inj_X$, $\DBL_X$, $\DBL_X^1$, $\Sym_X\cup\DBL_X$ and $\Surj_X$.  We will show that each of these semigroups has a finitely generated universal left congruence and finite left diameter.  To this end, we first establish the following mappings, which were introduced in \cite{Gallagher:2005} to prove that $\T_X$ and $\PT_X$ have monogenic diagonal left acts.

Choose any bijection $\tnu : X\to X\times X,$ and let $\ta=\tnu\pi_1$ and $\tb=\tnu\pi_2,$ where $\pi_1,\pi_2 : X\times X\to X$ denote the projections onto the first and second coordinates, respectively.  Note that $\ta,\tb\in\DBL_X$.  For each pair $\th,\phi\in\T_X,$ define a map
$$\tg(\th,\phi) : X\to X, x\mapsto(x\th, x\phi)\tnu^{-1}.$$
Observe that $\ker\tg(\th,\phi)=\ker\th\cap\ker\phi,$ and $\tg(\th,\phi)\ta=\th$ and $\tg(\th,\phi)\tb=\phi.$  It follows immediately that $\T_X\times\T_X=\T_X(\ta,\tb).$

We fix the maps $\tnu,$ $\ta,$ $\tb$ and $\tg(\th,\phi)$ ($\th,\phi\in\T_X$) for the remainder of this section.

\begin{defn}
Let $S$ be a subsemigroup of $\T_X$ such that $\ta, \tb\in S,$ let $\th,\phi\in S,$ and let $k\in\N.$  By an $(\ta,\tb,k)${\em-inducing sequence from $\th$ to $\phi$} (in $S$), we mean a sequence $$\th=\psi_1, \psi_2, \dots, \psi_{k+1}=\phi$$ of elements of $S$ where $\tg(\psi_i,\psi_{i+1})\in S$ for each $i\in\{1,\dots,k\}.$
\end{defn}

The following lemma is an analogue of the (1)$\Rightarrow$(2) part of Lemma \ref{lem:inducing,right}, showing that $(\ta,\tb,k)$-inducing sequences give rise to $(\ta,\tb)$-sequences of length $k$.

\begin{lem}\label{lem:inducing,left}
Let $S$ be a subsemigroup of $\T_X$ such that $\ta, \tb\in S.$  If there exists an $(\ta,\tb,k)$-inducing sequence
$$\th=\psi_1, \psi_2, \dots, \psi_{k+1}=\phi$$ 
from $\th$ to $\phi$ in $S,$ then there exists an $(\ta,\tb)$-sequence
$$\th=\g_1\ta,\, \g_1\tb=\g_2\ta,\, \dots,\, \g_{k-1}\tb=\g_k\ta,\, \g_k\tb=\phi$$
from $\th$ to $\phi$ of length $k$ in $S.$
\end{lem}

\begin{proof}
By definition we have $\tg(\psi_i,\psi_{i+1})\in S$ for each $i\in\{1,\dots,k\}.$  Letting $\g_i=\tg(\psi_i,\psi_{i+1}),$ we have $\g_i\ta=\psi_i$ and $\g_i\tb=\psi_{i+1}.$  Hence, we have an $(\ta,\tb)$-sequence
$$\th=\g_1\ta,\, \g_1\tb=\g_2\ta,\, \dots,\, \g_{k-1}\tb=\g_k\ta,\, \g_k\tb=\phi$$
in $S,$ as required.
\end{proof}

Lemma \ref{lem:inducing,left} yields the following counterpart of Proposition \ref{prop:rightdiameter}.

\begin{prop}\label{prop:leftdiameter}
Let $S$ be a subsemigroup of $\T_X$ such that:
\begin{enumerate}
\item $\ta,\tb\in S$; 
\item there exists $n\in\N$ such that for any pair $\th, \phi\in S$ there is an $(\ta,\tb,k)$-inducing sequence from $\th$ to $\phi$ in $S$ for some $k\leq n.$
\end{enumerate} 
Then $\w_S^l$ is generated by the pair $(\ta,\tb)$ and $D_l(S)\leq n.$  Furthermore, if $n=1$ (so that $\tg(\th,\phi)\in S$ for any $\th,\phi\in S$), then the diagonal left $S$-act is generated by $(\ta,\tb)$ (and is hence monogenic).
\end{prop}

We now consider $\T_X\!\setminus\Inj_X$.

\begin{thm}
The semigroup $\T_X\!\setminus\Inj_X$ has left diameter 2.
\end{thm}

\begin{proof}
Let $S=\T_X\!\setminus\Inj_X.$  By Table \ref{table:da} and Proposition \ref{prop:diagonal,diameter}, $S$ does not have left diameter 1.  Using Proposition \ref{prop:leftdiameter}, we show that $\w_S^l=\langle(\ta,\tb)\rangle$ with $D_l(S)\leq 2,$ and hence $D_l(S)=2.$  

It is clear that $\ta,\tb\in S.$  Fix any $x\in X,$ and consider arbitrary $\th,\phi\in S.$  Certainly $c_x\in S.$  We have
$$\ker\tg(\th,c_x)=\ker\th\cap\ker c_x=\ker\th\cap(X\times X)=\ker\th.$$  Therefore, since $\th$ is not injective, it follows that $\tg({\th,c_x})$ is not injective, i.e.\ $\tg(\th,c_x)\in S.$  Similarly, we have $\tg(c_x,\phi)\in S.$  Thus, there is an $(\ta,\tb,2)$-inducing sequence $\th,\,c_x,\,\phi,$ as required.
\end{proof}

We now turn our attention to the dual Baer-Levi semigroup $\DBL_X$. 

We call a partition of an infinite set $Y$ a $\DBL${\em-partition} of $Y$ if it is of the form $\{A_y : y\in Y\}$ where $|A_y|=|Y|$ for all $y\in Y.$  For each $\a\in\DBL_X$, the set $\{x\a^{-1} : x\in X\}$ of kernel classes of $\a$ forms a $\DBL$-partition of $X.$ 

The following technical lemma concerning $\DBL$-partitions will be crucial in determining the left diameter of $\DBL_X$.

\begin{lem}\label{lem:DBL}
Let $\{A_x : x\in X\}$ and $\{B_x : x\in X\}$ be a pair of $\DBL$-partitions of $X.$  Then there exists a third $\DBL$-partition $\{C_x : x\in X\}$ of $X$ such that for each $x\in X$ the set $\{A_x\cap C_y : y\in X\}$ is a $\DBL$-partition of $A_x$, and $\{B_x\cap C_y : y\in X\}$ is a $\DBL$-partition of $B_x$. 
\end{lem} 
 
\begin{proof} 
Let $|X|=\k$, and for convenience assume that $X=\k,$ where as usual the cardinal $\k$ is identified with the set of all ordinals $\lam<\k.$  So, consider a pair of $\DBL$-partitions $\{A_{\lam} : \lam\in\k\}$ and $\{B_{\lam} : \lam\in\k\}.$  We begin by defining a sequence $(x_\lam)_{\lam\in\k}$ of distinct elements of $X$ by transfinite induction, as follows.  First, we define the set
$$T=\k^3\times\{0,1\}=\{(\a,\b,\g,n) : \a,\b,\g\in\k,\ n\in\{0,1\}\}.$$
Since $|T|=\k$, we may fix a bijection $\k\to T, \lam\mapsto t_\lam$.  Now let $\lam\in\k$, and suppose that we have defined the elements $x_\mu$ for all $\mu<\lam$.  Also write $t_\lam=(\a,\b,\g,n)$, and define $Y=\{x_\mu : \mu<\lam\}.$  Since $|Y|=|\lam|<\k$ (as $\k$ is a cardinal), we can define $x_\lam$ to be any element of $A_\a\!\setminus\!Y$ if $n=0,$ or any element of $B_\a\!\setminus\!Y$ if $n=1.$  

Now that we have defined the sequence $(x_\lam)_{\lam\in\k}$, for each $\lam\in\k$ we define 
$$T_\lam=\k\times\k\times\{\lam\}\times\{0,1\}=\{(\a,\b,\lam,n) : \a,\b\in\k,\ n\in\{0,1\}\}.$$
Finally, we set
$$C_\lam=\begin{cases}
\{x_\mu : t_\mu\in T_\lam\}&\text{if $\lam\geq1$,}\\
\{x_\mu : t_\mu\in T_\lam\}\cup X\!\setminus\!\{x_\lam : \lam\in\k\}&\text{if $\lam=0.$}
\end{cases}$$
Then $\{C_\lam : \lam\in\k\}$ is a $\DBL$-partition of $X$ because $\{T_\lam : \lam\in\k\}$ is a $\DBL$-partition of $T.$  Also, for any $\lam,\mu\in\k,$ the set $C_\lam$ contains $\k$ elements of the form $x_\nu$ where $t_\nu\in\{\mu\}\times\k\times\{\lam\}\times\{0\},$ each of which belongs to $A_\mu$ by definition.  This shows that $|A_\mu\cap C_\lam|=\k$ for all $\lam,\mu\in\k.$  Similarly, we have $|B_\mu\cap C_\lam|=\k$ for all $\lam,\mu\in\k.$  This completes the proof.
\end{proof}

We are now in a position to compute the left diameter of $\DBL_X$.

\begin{thm}\label{thm:DBL} 
The dual Baer-Levi semigroup $\DBL_X$ has left diameter 2.
\end{thm}

\begin{proof}
Let $S=\DBL_X.$  By Table \ref{table:da} and Proposition \ref{prop:diagonal,diameter}, $S$ does not have left diameter 1.

To prove the inequality $D_l(S)\leq 2$ we use Proposition \ref{prop:leftdiameter}.
We have already noted that $\ta,\tb\in S.$
Consider $\th,\phi\in S.$  For each $x\in X,$ let $A_x=x\th^{-1}$ and $B_x=x\phi^{-1}.$  Then $\{A_x : x\in X\}$ and $\{B_x : x\in X\}$ are $\DBL$-partitions of $X.$  By Lemma \ref{lem:DBL}, there exists a $\DBL$-partition $\{C_x : x\in X\}$ of $X$ such that for each $x\in X$ the set $\{A_x\cap C_y : y\in X\}$ is a $\DBL$-partition of $A_x$, and $\{B_x\cap C_y : y\in X\}$ is a $\DBL$-partition of $B_x$.  Let $\lam\in S$ be given by $x\lam^{-1}=C_x$ for all $x\in X.$  We claim that 
$$\th,\, \lam,\, \phi\,$$ is an $(\ta,\tb,2)$-inducing sequence from $\th$ to $\phi.$   Letting $\g_1=\tg(\th,\lam)$ and $\g_2=\tg(\lam,\phi),$   we need to show that $\g_1,\g_2\in S.$  Indeed, for each $y\in X$ we have
\begin{align*}
y\g_1^{-1}&=\{x\in X : (x\th,x\lam)=y\tnu\}
=\{x\in X : x\th=y\ta\text{ and }x\lam=y\tb\}\\
&=\{x\in X : x\in(y\ta)\th^{-1}\text{ and }x\in(y\tb)\lam^{-1}\}
=\{x\in X : x\in A_{y\ta}\text{ and }x\in C_{y\tb}\}\\
&=A_{y\ta}\cap C_{y\tb},
\end{align*}
and similarly $y\g_2^{-1}=C_{y\ta}\cap B_{y\tb}.$
Thus, for each $y\in X$ we have
$$|y\g_1^{-1}|=|A_{y\ta}\cap C_{y\tb}|=|X|\quad\text{and}\quad|y\g_2^{-1}|=|C_{y\ta}\cap B_{y\tb}|=|X|,$$ so that $\g_1,\g_2\in S,$ as required.  
\end{proof}

Next, we establish a technical lemma, and then employ it to show that submonoids of $\Surj_X$ containing $\DBL_X^1$ have left diameter either 3 or 4.  In this lemma and what follows, a subset $Y\subset X$ is {\em colarge} (in $X$) if $|X\!\setminus\!Y|=|X|$.  

\begin{lem}\label{lem:colarge}
Let $\{\a_1,\dots,\a_n\}$ be a finite subset of $\T_X$, and let $x_1,\dots,x_{n-1}$ be (not necessarily distinct) elements of $X.$  If the set $Y=\bigcup_{1\leq i\leq n-1}x_i\a_i^{-1}$ is colarge in $X,$ then there exists at most one element $x\in X$ such that $Y\cup x\a_n^{-1}$ is not colarge in $X.$ 
\end{lem} 
 
\begin{proof}
Suppose that $Y\cup x\a_n^{-1}$ is not colarge.  Since $X\!\setminus\!Y=\big(X\!\setminus\!(Y\cup x\a_n^{-1})\big)\cup(x\a_n^{-1}\!\setminus\!Y)$, and $|X\!\setminus\!Y|=|X|,$ it follows that $|x\a_n^{-1}\!\setminus\!Y|=|X|.$  But then for any $y\in X\!\setminus\!\{x\}$ we have 
$$X\!\setminus\!(Y\cup y\a_n^{-1})=(X\!\setminus\!Y)\cap(X\!\setminus\!y\a_n^{-1})\supseteq x\a_n^{-1}\!\setminus\!Y,$$
and hence $Y\cup y\a_n^{-1}$ is colarge.
\end{proof}

\begin{prop}\label{prop:DBL1,Surj} 
If $S$ is a monoid such that $\DBL_X^1\leq S\leq\Surj_X$, then $\w_S^l$ is finitely generated and $D_l(S)\in\{3,4\}.$
\end{prop}

\begin{proof}
Since $\DBL_X$ is an ideal of $S,$ it follows from Theorem \ref{thm:DBL} and the dual of Lemma \ref{lem:rightideal} that $\w_S^l$ is finitely generated with $D_l(S)\leq 4.$

Now suppose for a contradiction that $D_l(S,U)\leq2$ for some finite set $U\subset S$.  Let 
$$V=\{\mu^{-1}\lam : \mu\in U\cap\Sym_X, \lam\in U\},$$ 
and note that $V$ is finite.  By an easy induction argument, using Lemma \ref{lem:colarge}, we may fix elements $y_\phi\in X$ ($\phi\in V$) such that $A=\bigcup_{\phi\in V}y_\phi\phi^{-1}$ is colarge.  For each pair $\a,\b\in U$ and each $\phi\in V,$ we fix some $x_{\a,\b,\phi}\in X$ such that $x_{\a,\b,\phi}\a^{-1}\cap y_{\phi}\b^{-1}\neq\emptyset$ (since $y_\phi\in X=\im\b,$ we can pick any $z\in y_\phi\b^{-1}$ and define $x_{\a,\b,\phi}=z\a$).  Now choose any $\th\in\DBL_X\,(\subseteq S)$ such that $\{x_{\a,\b,\phi} : \a,\b\in U,\phi\in V\}\th^{-1}\subseteq X\!\setminus\!A$ (such a map exists because $A$ is colarge).  As $D_l(S,U)\leq2,$ there exists a $U$-sequence
$$\th=\g\a,\,\g\b=\d\lam,\,\d\mu=1_X$$
(where $\a,\b,\lam,\mu\in U$).
Since $\Surj_X\!\setminus\!\Sym_X$ is an ideal of $\Surj_X$ \cite[Proof of Theorem 4.4.2]{Gallagher:thesis}, it follows that $\d,\mu\in\Sym_X$ with $\d=\mu^{-1}$.  Thus, letting $\phi=\mu^{-1}\lam\in V,$ we have $\th=\g\a$ and $\g\b=\phi.$  Let $x=x_{\a,\b,\phi}$, choose some $z\in x\a^{-1}\cap y_{\phi}\b^{-1}$, and then pick some $u\in z\g^{-1}$.  Then $u\th=u\g\a=z\a=x$, so $u\in x\th^{-1}$.  On the other hand, we have $u\phi=u\g\b=z\b=y_{\phi},$ so $u\in y_\phi\phi^{-1}\in A$.  But this contradicts the fact that $x\th^{-1}\subseteq X\!\setminus\!A.$  Thus $D_l(S)\geq3.$
\end{proof}

Note that the set $\Surj_X\!\setminus\!\DBL_X$ is not a subsemigroup of $\Surj_X$ (in contrast to the situation for $\Inj_X$, where $\Inj_X\!\setminus\!\BL_X$ {\em is} a subsemigroup).  However, for a subsemigroup $S$ of $\Surj_X$, if $S\!\setminus\!\DBL_X$ is finite and non-empty then it is a subgroup of $\Sym_X$.

\begin{thm}\label{thm:DBL1} 
For any finite subgroup $G$ of $\Sym_X$, the monoid $G\cup\DBL_X$ has left diameter 3.  In particular, $\DBL_X^1$ has left diameter 3.
\end{thm}

\begin{proof}
Let $S=G\cup\DBL_X$.  Since $D_l(\DBL_X,\{\ta,\tb\})=2$ (by the proof of Theorem \ref{thm:DBL}), it is clear that $\w_S^l$ is generated by the finite set $U=G\cup\{\ta,\tb\}$ and that $D_l(S,U)\leq3.$  On the other hand, we have $D_l(S)\geq3$ by Proposition \ref{prop:DBL1,Surj}.  Thus $D_l(S)=3$.
\end{proof}

We now raise the following question, concerning a natural analogue of Proposition \ref{prop:S,Inj}.

\begin{op}\label{op}
If $S$ is a subsemigroup of $\Surj_X$ containing a finitely transitive subsemigroup of $S\!\setminus\!\DBL_X$, and $\w_S^l$ is finitely generated, is $D_r(S)\geq4$?
\end{op}

The following result affirmatively answers Open Problem \ref{op} in the special case that $S$ contains $\Sym_X$.

\begin{prop}\label{prop:S,Surj}
Let $S$ be a monoid such that $\Sym_X\leq S\leq\Surj_X$.  If $\w_S^l$ is finitely generated, then $D_l(S)\geq4$.  
\end{prop}  

\begin{proof}
Suppose for a contradiction that $D_l(S,U)\leq3$ for some finite set $U\subset S$. 
Let $P$ denote the collection of all tuples $(\a_1,\b_1,\a_2,\b_2,\a_3,\b_3)\in U^6$ where $\a_1,\b_3\in \Sym_X$.  Write $P=\{p_1,\dots,p_n,\dots,p_{n+r}\},$ where $p_i=(\a_1^{(i)},\dots, \b_3^{(i)})$ ($1\leq i\leq n+r$) and $p_1,\dots,p_n$ are those tuples $p=(\a_1,\dots,\b_3)\in P$ for which there exists an element $w_p\in X$ such that the set $\bigcup_{z\neq w_p}z\b_2^{-1}\a_2$ is finite. 
For $i\in\{1,\dots,n\},$ write $w_{p_i}$ as $w_i$, and let $W=\{w_1,\dots,w_n\}.$  Also, let $L=\bigcup_{1\leq i\leq n}\bigcup_{z\neq w_i}z(\b_2^{(i)})^{-1}\a_2^{(i)}$, and note that $L$ is finite. 
We now establish the following claim.
\begin{claim*}
(1) There exist $u_i,v_i,x_i,y_i\in X$ ($1\leq i\leq n$) such that
$$u_i\notin W,\;v_i\notin L,\;x_i\in u_i(\a_3^{(i)})^{-1}\b_3^{(i)}\;\text{ and }\;y_i\in v_i(\b_1^{(i)})^{-1}\a_1^{(i)},$$
with $x_i\neq x_{i'}$ and $y_i\neq y_{i'}$ for distinct pairs $i,i'\in\{1,\dots,n\}$.\\
(2) Let $y_i$ ($1\leq i\leq n$) be as given in (1), and let $$K=\{y_i(\a_1^{(n+j)})^{-1}\b_1^{(n+j)}: 1\leq i\leq n, 1\leq j\leq r\}.$$
For each $1\leq j\leq r,$ there exist $a_j,b_j\in X$ with $b_j\in a_j(\b_2^{(n+j)})^{-1}\a_2^{(n+j)}\!\setminus\!K$ such that
$$\bigcup_{1\leq q\leq j}a_q(\a_3^{(n+q)})^{-1}\b_3^{(n+q)}\quad\text{and}\quad\bigcup_{1\leq q\leq j}b_q(\b_1^{(n+q)})^{-1}\a_1^{(n+q)}$$ are colarge.
\end{claim*}

\begin{proof}
We prove both (1) and (2) by induction.

(1) For the base case, pick any $u_1\notin W$ and $v_1\notin L,$ and then choose $x_1\in u_1(\a_3^{(1)})^{-1}\b_3^{(1)}$ and $y_1\in v_1(\b_1^{(1)})^{-1}\a_1^{(1)}$.

Now let $k\in\{2,\dots,n\},$ and assume that $u_i,v_i,x_i,y_i\in X$ ($1\leq i\leq k-1$) have been chosen such that 
$$u_i\notin W,\;v_i\notin L,\;x_i\in u_i(\a_3^{(i)})^{-1}\b_3^{(i)}\;\text{ and }\;y_i\in v_i(\b_1^{(i)})^{-1}\a_1^{(i)},$$
with $x_i\neq x_{i'}$ and $y_i\neq y_{i'}$ for distinct pairs $i,i'\in\{1,\dots,k-1\}$.
Since the map $(\b_3^{(k)})^{-1}\a_3^{(k)}$ is surjective, its set of kernel classes $\{x(\a_3^{(k)})^{-1} \b_3^{(k)} : x\in X\}$ is infinite.  Therefore, we may choose $u_k\notin W$ such that 
$$u_k(\a_3^{(k)})^{-1}\b_3^{(k)}\cap\{x_1,\dots,x_{k-1}\}=\emptyset.$$  Similarly, we may choose $v_k\notin L$ such that 
$$v_k(\b_1^{(k)})^{-1}\a_1^{(k)}\cap\{y_1,\dots,y_{k-1}\}=\emptyset.$$
Take any $x_k\in u_k(\a_3^{(k)})^{-1}\b_3^{(k)}$ and $y_k\in v_k(\b_1^{(k)})^{-1}\a_1^{(k)}$.  This completes the inductive step.\\~\\
(2) We first note that $K$ is finite.  Now, for the base case, let $a_1\in X$ be such that $a_1(\b_2^{(n+1)})^{-1}\a_2^{(n+1)}$ is not contained in $K$ (which is possible by surjectivity), and let $b_1\in a_1(\b_2^{(n+1)})^{-1}\a_2^{(n+1)}\!\setminus\!K.$  The sets $a_1(\a_3^{(n+1)})^{-1}\b_3^{(n+1)}$ and $b_1(\b_1^{(n+1)})^{-1}\a_1^{(n+1)}$ are colarge since the sets $a_1(\a_3^{(n+1)})^{-1}$ and $b_1(\b_1^{(n+1)})^{-1}$ are colarge and the maps $\b_3^{(n+1)}$ and $\a_1^{(n+1)}$ are bijections.

Now let $j\in\{2,\dots,r\},$ and assume that we have chosen $a_q,b_q\in X$ ($1\leq q\leq j-1$) with 
$b_q\in a_q(\b_2^{(n+q)})^{-1}\a_2^{(n+q)}\!\setminus\!K$ 
such that $$A_{j-1}=\bigcup_{1\leq q\leq j-1}a_q(\a_3^{(n+q)})^{-1}\b_3^{(n+q)}\quad\text{and}\quad B_{j-1}=\bigcup_{1\leq q\leq k-1}b_q(\b_1^{(n+q)})^{-1}\a_1^{(n+q)}$$
are colarge.  Observing that for any $\a,\b\in\B_X$ we have $\a^{-1}\b=(\b^{-1}\a)^{-1}$, by Lemma \ref{lem:colarge} there exists at most one element $c_j$ such that the set 
$A_{j-1}\cup c_j(\a_3^{(n+j)})^{-1}\b_3^{(n+j)}$
is {\em not} colarge, and at most one element $d_j$ such that $B_{j-1}\cup d_j(\b_1^{(n+j)})^{-1}\a_1^{(n+j)}$ is not colarge.  Let $K'=K\cup\{d_j\}$ if $d_j$ exists; otherwise, let $K'=K.$
Now, if $c_j$ were the only element of $X$ such that 
$$c_j(\b_2^{(n+j)})^{-1}\a_2^{(n+j)}\not\subseteq K',$$
then we would have $\bigcup_{z\neq c_j}z(\b_2^{(n+j)})^{-1}\a_2^{(n+j)}\subseteq K',$ which is finite, contradicting the fact that $p_{n+j}$ is not one of the tuples $p_1,\dots,p_n$. 
Hence, we may pick $a_j\in X$ (with $a_j\neq c_j$ if $c_j$ exists) such that
$$A_{j-1}\cup a_j(\a_3^{(n+j)})^{-1}\b_3^{(n+j)}=\bigcup_{1\leq q\leq j}a_q(\a_3^{(n+q)})^{-1}\b_3^{(n+q)}$$ 
is colarge, and $a_j(\b_2^{(n+j)})^{-1}\a_2^{(n+j)}$ possesses an element $b_j\notin K'$.  Then
$$B_{j-1}\cup b_j(\b_1^{(n+j)})^{-1}\a_1^{(n+j)}=\bigcup_{1\leq q\leq j}b_q(\b_1^{(n+q)})^{-1}\a_1^{(n+q)}$$ 
is colarge.  This completes the inductive step.  
\end{proof}
 
We fix the elements $u_i,v_i,x_i,y_i,a_j,b_j\in X$ ($1\leq i\leq n$, $1\leq j\leq r$) and the set $K,$ as given in the above claim, for the remainder of this proof.  Let 
$$A=\bigcup_{1\leq j\leq r}a_j(\a_3^{(n+j)})^{-1}\b_3^{(n+j)}\qquad\text{and}\qquad B=\bigcup_{1\leq j\leq r}b_j(\b_1^{(n+j)})^{-1}\a_1^{(n+j)}.$$
Choose $\th\in\Sym_X$ such that 
$$x_i\th=y_i\;\;(1\leq i\leq n)\quad\text{and}\quad(A\!\setminus\!\{x_i : 1\leq i\leq n\})\th\subseteq X\!\setminus\!B.$$ 
(Such a bijection exists since the elements $x_1,\dots,x_n$ are distinct, the elements $y_1,\dots,y_n$ are distinct, and the sets $A\!\setminus\!\{x_i : 1\leq i\leq n\}$ and $B$ are both colarge.)  Since $D_l(S,U)\leq3$, there exists a $U$-sequence
$$\th=\g_1\a_1,\;\g_1\b_1=\g_2\a_2,\;\g_2\b_2=\g_3\a_3,\;\g_3\b_3=1_X.$$
As $\Surj_X\!\setminus\!\Sym_X$ is an ideal of $\Surj_X$, we have $\g_1,\a_1,\g_3,\b_3\in\Sym_X$ with $\g_1=\th\a_1^{-1}$ and $\g_3=\b_3^{-1}$.  Thus $p=(\a_1,\b_1,\a_2,\b_2,\a_3,\b_3)\in P,$ and we have
$$\th\a_1^{-1}\b_1=\g_2\a_2,\;\g_2\b_2=\b_3^{-1}\a_3.$$
Now, for any $z\in X,$
\begin{align*}
y\in z\b_2^{-1}\a_2&\Leftrightarrow\text{there exists }x\in X\text{ such that }x\b_2=z, y=x\a_2\\
&\Leftrightarrow\text{there exists }x'\in X\text{ such that }x'\g_2\b_2=z, y=x'\g_2\a_2\:\:\text{(since $\g_2$ is surjective)}\\
&\Leftrightarrow y\in z(\g_2\b_2)^{-1}(\g_2\a_2)\\
&\Leftrightarrow y\in z(\b_3^{-1}\a_3)^{-1}(\th\a_1^{-1}\b_1)\\
&\Leftrightarrow y\in z\a_3^{-1}\b_3\th\a_1^{-1}\b_1. 
\end{align*}
Thus, for each $z\in X$ we have
\begin{equation}\label{eq:surj}
z\b_2^{-1}\a_2=z\a_3^{-1}\b_3\th\a_1^{-1}\b_1. 
\end{equation}
Suppose first that $p=p_i$ where $i\in\{1,\dots,n\}$ (so $\a_k=\a_k^{(i)},$ $\b_k=\b_k^{(i)}$ for $k=1,2,3$).
Then $u_i\notin W,$ so $u_i\b_2^{-1}\a_2\subseteq L.$
But then we have
$$v_i=y_i\a_1^{-1}\b_1=x_i\th\a_1^{-1}\b_1\in u_i \a_3^{-1}\b_3\th\a_1^{-1}\b_1=u_i\b_2^{-1}\a_2\subseteq L,$$
contradicting the choice of $v_i$.

Now suppose that $p=p_{n+j}$ where $j\in\{1,\dots,r\}$ (so $\a_k=\a_k^{(n+j)},$ $\b_k=\b_k^{(n+j)}$ for $k=1,2,3$).  Then we have
\begin{align*} 
b_j\in a_j\b_2^{-1}\a_2=a_j\a_3^{-1}\b_3\th\a_1^{-1}\b_1
&\subseteq A\th\a_1^{-1}\b_1\\
&\subseteq\big((A\!\setminus\!\{x_i : 1\leq i\leq n\})\cup\{x_i : 1\leq i\leq n\}\big)\th\a_1^{-1}\b_1\\
&\subseteq(X\!\setminus\!B)\a_1^{-1}\b_1\cup\{y_i\a_1^{-1}\b_1 : 1\leq i\leq n\}\\
&\subseteq(X\!\setminus\!b_j\b_1^{-1}\a_1)\a_1^{-1}\b_1\cup K\\
&\subseteq(X\!\setminus\!\{b_j\})\cup K=X\!\setminus\!\{b_j\},
\end{align*}
where the first equality is due to \eqref{eq:surj} and the final equality is due to the fact that $b_j\notin K.$  Again, we have a contradiction.  Thus $D_l(S)\geq4$.
\end{proof}

By Propositions \ref{prop:DBL1,Surj}  and \ref{prop:S,Surj}, we have:

\begin{thm}\label{thm:S,Surj}
If $S$ is a monoid such that $\Sym_X\cup\DBL_X\leq S\leq\Surj_X,$ then $D_l(S)=4.$  In particular, the monoids $\Sym_X\cup\DBL_X$ and $\Surj_X$ have left diameter 4.
\end{thm}

Unfortunately, we have not obtained an analogue of Theorem \ref{thm:Sym,Inj}, classifying those subsemigroups of $\Surj_X$ containing $\Sym_X$ that have left diameter 4.  We conclude this section by considering a potential such classification.

First, for a map $\a\in\T_X$ define 
$$K(\a)=\{x\in X: |x\a^{-1}|=|X|\},$$
and let $k(\a)=|K(\a)|.$  The set $K(\a)$ and parameter $k(\a)$ were introduced in \cite{Howie:1998}.
Note that $\DBL_X=\{\a\in\Surj_X : K(\a)=X\}.$

\begin{prop}\label{prop:containDBL}
Let $S$ be a monoid such that $\Sym_X\leq S\leq\Surj_X$.  Then $S$ contains $\DBL_X$ if and only if there exists some $\a\in S$ such that $k(\a)=|X|.$
\end{prop}

\begin{proof}
We have already observed that $K(\a)=X$ for any $\a\in\DBL_X$, so the forward direction clearly holds.

For the reverse implication, let $Y$ be any subset of $K(\a)$ such that $|Y|=|X\!\setminus\!Y|=|X|$, and fix a bijection $X\to Y, x\mapsto y_x$.  For each $x\in X$ fix some $a_x\in x\a^{-1}$, and note that the set $A=\{a_x : x\in X\}$ satisfies $|A|=|X\!\setminus\!A|=|X|$, as $K(\a)\not=\emptyset$.  The map $Y\to A, y_x\mapsto a_x$ can therefore be extended to a bijection $\pi\in\Sym_X$.  We then see that $\b=\a\pi\a\in S$ belongs to $\DBL_X$.  Indeed, for any $x\in X$ we have $x\b^{-1}=x\a^{-1}\pi^{-1}\a^{-1}\supseteq a_x\pi^{-1}\a^{-1}=y_x\a^{-1}$, and $|y_x\a^{-1}|=|X|$ as $y_x\in Y\sub K(\a)$.  

Finally, let $\g\in\DBL_X$ be arbitrary.  Take any bijection $\psi\in\Sym_X\subseteq S$ such that $(x\g^{-1})\psi=x\b^{-1}$ for each $x\in X$.  Then, for each $x\in X,$ we have $x\psi\in\big((x\g)\g^{-1})\psi=(x\g)\b^{-1},$ so that $x\g=(x\psi)\b=x(\psi\b).$  Thus $\g=\psi\b\in S,$ and hence $\DBL_X\subseteq S.$
\end{proof}

\begin{op}
For a monoid $S$ such that $\Sym_X\leq S\leq\Surj_X$, are the following equivalent?
\begin{enumerate}
\item $D_l(S)=4$;
\item $\w_S^l$ is finitely generated;
\item there exists $\a\in S$ such that $k(\a)=|X|$;
\item $S$ contains $\DBL_X$.
\end{enumerate}
(1)$\Rightarrow$(2) certainly holds, we have (3)$\Leftrightarrow$(4) by Proposition \ref{prop:containDBL}, and (4)$\Rightarrow$(1) follows from Theorem \ref{thm:S,Surj}.  Thus, to answer this question in the affirmative, it would suffice to prove that (2) implies (3).
\end{op}

\section{Monoids of Partitions}
\label{sec:partition}

In this section we consider the partition monoid $\P_X$ and the partial Brauer monoid $\PB_X$ (where $X$ is an arbitrary infinite set).

The {\em partition monoid} $\P_X$ consists of all set partitions of $X\cup X'$, where $X'=\{x' : x\in X\}$ is a disjoint copy of $X.$  So, an element of $\P_X$ is of the form $\a=\{A_i : i\in I\},$ where the $A_i$ are non-empty, pairwise disjoint subsets of $X\cup X'$ such that $X\cup X'=\bigcup_{i\in I}A_i$; the $A_i$ are called the {\em blocks} of $\a.$    
An element of $\P_X$ may be represented as a graph with vertices $X\cup X'$ whose connected components are the blocks of the partition; when depicting such a graph, vertices from $X$ and $X'$ are displayed on upper and lower rows, respectively.  It is from this graph-theoretic viewpoint that we define the product in $\P_X$.

Let $\a,\b\in\P_X$.  Introduce another copy $X''=\{x'' : x\in X\}$ of $X,$ disjoint from $X\cup X'$.   Denote by $\a_{\downarrow}$ the graph obtained from $\a$ by replacing every $x'$ with $x''$, and denote by $\b^{\uparrow}$ the graph obtained from $\b$ by replacing every $x$ with $x''$.  The product graph $\Pi(\a,\b)$ is the graph with vertex set $X\cup X''\cup X'$ and edge set the union of the edge sets of $\a_{\downarrow}$ and $\b^{\uparrow}$.  The graph $\Pi(\a,\b)$ is drawn with vertices from $X''$ displayed in a new middle row.
We define $\a\b$ to be the partition of $X\cup X'$ such that $u,v\in X\cup X'$ belong to the same block if and only if $u$ and $v$ belong to the same connected component of $\Pi(a,b)$.  We illustrate this product in Figure \ref{fig:P_X} (using elements from a finite partition monoid).  Under this product, $\P_X$ is a monoid with identity $\big\{\{x,x'\} : x\in X\big\}.$ 

The {\em partial Brauer monoid} $\PB_X$ is the submonoid of $\P_X$ consisting of all partitions whose blocks have size at most 2.  (Thus, the partition $\b$ in Figure \ref{fig:P_X} in fact belongs to $\PB_6$.)

\begin{figure}[h]
\begin{center}
\begin{tikzpicture}[scale=.5]
\begin{scope}[shift={(0,0)}]
\uvs{1,...,6}
\lvs{1,...,6}
\uarcx14{.6}
\uarcx23{.3}
\uarcx45{.3}
\uarcx56{.3}
\darcx26{.6}
\darcx45{.3}
\stline34
\draw(0.6,1)node[left]{$\a=$};
\draw[->](7.5,-1)--(9.5,-1);
\end{scope}
\begin{scope}[shift={(0,-4)}]
\uvs{1,...,6}
\lvs{1,...,6}
\uarc13
\uarc56
\darc45
\stline21
\stline43
\draw(0.6,1)node[left]{$\b=$};
\end{scope}
\begin{scope}[shift={(10,-1)}]
\uvs{1,...,6}
\lvs{1,...,6}
\uarcx14{.6}
\uarcx23{.3}
\uarcx45{.3}
\uarcx56{.3}
\darcx26{.6}
\darcx45{.3}
\stline34
\draw[->](7.5,0)--(9.5,0);
\end{scope}
\begin{scope}[shift={(10,-3)}]
\uvs{1,...,6}
\lvs{1,...,6}
\uarc13
\uarc56
\darc45
\stline21
\stline43
\end{scope}
\begin{scope}[shift={(20,-2)}]
\uvs{1,...,6}
\lvs{1,...,6}
\uarcx14{.6}
\uarcx23{.3}
\uarcx45{.3}
\uarcx56{.3}
\darc13
\darc45
\stline33
\draw(6.4,1)node[right]{$=\a\b$};
\end{scope}
\end{tikzpicture}
\caption{Two partitions $\a,\b\in\P_6$ (left), the product graph $\Pi(a,b)$ (middle) and the product $\a\b$ (right).}
\label{fig:P_X}
\end{center}
\end{figure}
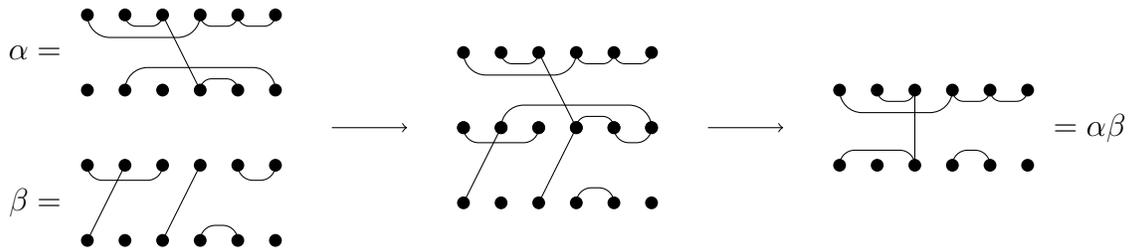

It turns out that both $\P_X$ and $\PB_X$ have right diameter 1 and left diameter 1, as follows from the following stronger result.

\begin{thm}\label{thm:P_X}
The diagonal right act and diagonal left act of both $\P_X$ and $\PB_X$ are monogenic.  Consequently, both $\P_X$ and $\PB_X$ have right diameter 1 and left diameter 1.
\end{thm}

\begin{proof}
The partition monoid $\P_X$ has an involution $*: \P_X\to\P_X,$ $\a\to\a^*$, where $\a^*$ is obtained from $\a$ by swapping dashed and undashed vertices (pictorially, $\a^*$ is `$\a$ upside-down'), and this map restricts to an involution on $\PB_X$.  It follows that for all partitions $\th,\phi,\a,\b,\g\in\P_X$ (or $\PB_X$) we have 
$$(\th,\phi)=(\a,\b)\g\;\Leftrightarrow\;(\th^*,\phi^*)=\g^*(\a^*,\b^*),$$
and hence the diagonal right act of $\P_X$ (resp.\ $\PB_X$) is monogenic if and only if the diagonal left act of $\P_X$ (resp.\ $\PB_X$) is monogenic.
Thus, it suffices to show that the diagonal right acts of $\P_X$ and $\PB_X$ are monogenic.

Divide $X$ into five subsets as follows:
$$X=A\sqcup B\sqcup C\sqcup D\sqcup E \qquad\text{where}\qquad|A|=|B|=|C|=|D|=|E|=|X|.$$
Write
$$
A=\{a_x : x\in X\},\quad B=\{b_x :x\in X\},\quad\text{etc}.$$
Define (and fix) $\a,\b\in\P_X$ as follows:
$$\a=\big\{\{x,a_x'\} : x\in X\}\cup\big\{\{b_x',c_x'\} : x\in X\big\}\cup\big\{\{d_x'\} : x\in X\big\}\cup \big\{\{e_x'\} : x\in X\big\};$$
$$\b=\big\{\{x,e_x'\big\} : x\in X\}\cup\big\{\{c_x',d_x'\} : x\in X\big\}\cup\big\{\{a_x'\} : x\in X\big\}\cup\big\{\{b_x'\} : x\in X\big\}.$$
See Figure \ref{fig:1}, and note that in fact $\a$ and $\b$ belong to $\PB_X$.

For each set $T=U\cup V'$ where $U,V\subseteq X,$ define
\[\lam(T)=\{a_u : u\in U\}\cup\{b_v : v\in V\}\qquad\text{and}\qquad\rho(T)=\{e_u : u\in U\}\cup\{d_v : v\in V\}.\]
Consider $\th,\phi\in\P_X$.  Set
\[\lam(\th)=\big\{\lam(T) : T\in\th\big\}\qquad\text{and}\qquad\rho(\phi)=\big\{\rho(T) : T\in\phi\big\}.\]
Note that $\lam(\th)$ and $\rho(\phi)$ are partitions of $A\cup B$ and $D\cup E$, respectively.  Now define
\[\g=\g(\th,\phi)=\lam(\th)\cup\big\{\{c_x,x'\} : x\in X\big\}\cup\rho(\phi).\]
Then $\g\in\P_X$ and $(\th,\phi)=(\a,\b)\g$; see Figures \ref{fig:2} and \ref{fig:3}.  Moreover, if $\th,\phi\in\PB_X$ then $\g\in\PB_X$.  Thus, we have $\P_X\times\P_X=(\a,\b)\P_X$ and $\PB_X\times\PB_X=(\a,\b)\PB_X$.
\end{proof}

\begin{figure}[H]
\begin{center}
\scalebox{0.75}{
\begin{tikzpicture}[scale=0.38]
\begin{scope}
\foreach \x/\y in {0/A,1/B,2/C,3/D,4/E} {\draw[|-|] (5*\x,-1.5) -- (5*\x+4,-1.5); \node () at (5*\x+2,-2) {$\y$};}
\fill[blue!20] (0,0)--(4,0)--(24,15)--(0,15)--(0,0);
{\nc\p5 \nc\q{1} \nc\rr{14} 
\fill[blue!20] (\p,0) arc(180:90:\q) (\p+\q,\q)--(\rr-\q,\q) arc(90:0:\q) -- (\p,0);
}
{\nc\p9 \nc\q{.3} \nc\rr{10} 
\fill[white] (\p,0) arc(180:90:\q) (\p+\q,\q)--(\rr-\q,\q) arc(90:0:\q) -- (\p,0);
}
\foreach \x in {0,1,2,3} {\lv{15+0.5*\x} \lv{20+0.5*\x}}
\foreach \x in {0,1} {\draw[dotted](17+5*\x,0)--(19+5*\x,0);}
\fill (10,15)circle(.17);
\lv1
\draw(10,15)--(1,0);
\lv6
\lv{13}
\darcx6{13}{.6}
\node () at (10,15.5) {$x$};
\node () at (1,-.7) {$a_x$};
\node () at (6,-.7) {$b_x$};
\node () at (13,-.7) {$c_x$};
\end{scope}
\begin{scope}[shift={(30,0)}]
\foreach \x/\y in {0/A,1/B,2/C,3/D,4/E} {\draw[|-|] (5*\x,-1.5) -- (5*\x+4,-1.5); \node () at (5*\x+2,-2) {$\y$};}
\fill[blue!20] (24-0,0)--(24-4,0)--(24-24,15)--(24-0,15)--(24-0,0);
{\nc\p{24-14} \nc\q{1} \nc\rr{24-5} 
\fill[blue!20] (\p,0) arc(180:90:\q) (\p+\q,\q)--(\rr-\q,\q) arc(90:0:\q) -- (\p,0);
}
{\nc\p{24-10} \nc\q{.3} \nc\rr{24-9} 
\fill[white] (\p,0) arc(180:90:\q) (\p+\q,\q)--(\rr-\q,\q) arc(90:0:\q) -- (\p,0);
}
\foreach \x in {0,1,2,3} {\lv{0+0.5*\x} \lv{5+0.5*\x}}
\foreach \x in {0,1} {\draw[dotted](2+5*\x,0)--(4+5*\x,0);}
\fill (10,15)circle(.17);
\lv{23}
\draw(10,15)--(23,0);
\lv{11}
\lv{18}
\darcx{11}{18}{.6}
\node () at (10,15.5) {$x$};
\node () at (23,-.7) {$e_x$};
\node () at (18,-.7) {$d_x$};
\node () at (11,-.7) {$c_x$};
\end{scope}
\end{tikzpicture}}
\vspace{-1em}
\caption{The partitions $\a,\b\in\PB_X$ (left and right, respectively).}
\label{fig:1}
\end{center}
\end{figure}
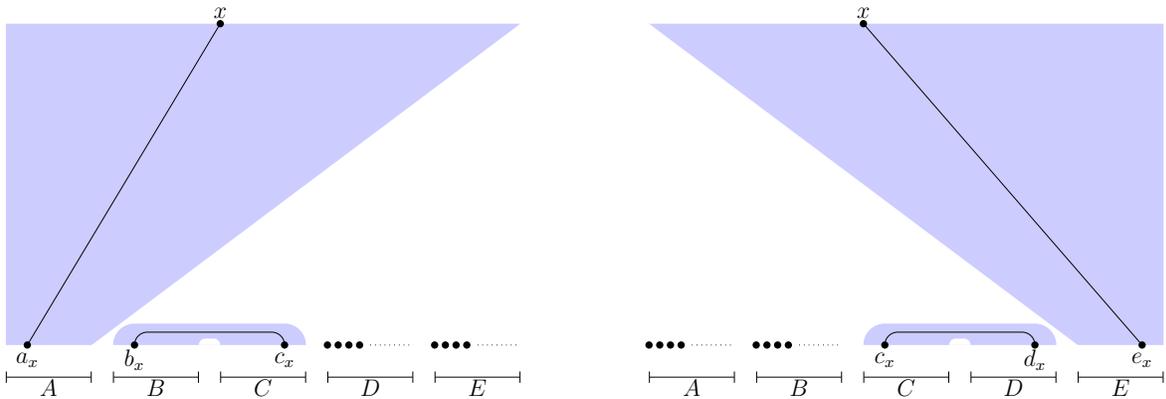

\begin{figure}[H]
\begin{center}
\scalebox{0.75}{
\begin{tikzpicture}[scale=0.38]
\begin{scope}[shift={(0,-45)}]
\foreach \x/\y in {0/A,1/B,2/C,3/D,4/E} {\draw[|-|] (5*\x,16.5) -- (5*\x+4,16.5); \node () at (5*\x+2,17) {$\y$};}
\fill[red!20] (0,0)--(24,0)--(14,15)--(10,15)--(0,0);
\fill (12,15)circle(.17);
\lv{7}
\draw(12,15)--(7,0);
\node () at (12,15.5) {$c_x$};
\node () at (7,-.7) {$x$};
{\nc\p{0} \nc\q{2} \nc\rr{9} 
\fill[red!20] (\p,15) arc(180:270:\q) (\p+\q,15-\q)--(\rr-\q,15-\q) arc(270:360:\q) -- (\p,15);
}
{\nc\p{4} \nc\q{.4} \nc\rr{5} 
\fill[white] (\p,15) arc(180:270:\q) (\p+\q,15-\q)--(\rr-\q,15-\q) arc(270:360:\q) -- (\p,15);
}
{\nc\p{15+0} \nc\q{2} \nc\rr{15+9} 
\fill[red!20] (\p,15) arc(180:270:\q) (\p+\q,15-\q)--(\rr-\q,15-\q) arc(270:360:\q) -- (\p,15);
}
{\nc\p{15+4} \nc\q{.4} \nc\rr{15+5} 
\fill[white] (\p,15) arc(180:270:\q) (\p+\q,15-\q)--(\rr-\q,15-\q) arc(270:360:\q) -- (\p,15);
}
\node () at (4.5,13.8) {$\lam(\th)$};
\node () at (15+4.5,13.8) {$\rho(\phi)$};
\end{scope}
\end{tikzpicture}}
\vspace{-1em}
\caption{The partition $\g=\g(\th,\phi)\in\P_X$.}
\label{fig:2}
\end{center}
\end{figure}

\begin{figure}[H]
\begin{center}
\scalebox{0.75}{
\begin{tikzpicture}[scale=0.3]
\begin{scope}[shift={(0,-15)}]
\fill[red!20] (0,0)--(24,0)--(14,15)--(10,15)--(0,0);
\draw[|-|] (-1,15)--(-1,0);
\node () at (-1.6,7.5) {$\g$};
{\nc\p{0} \nc\q{2} \nc\rr{9} 
\fill[red!20] (\p,15) arc(180:270:\q) (\p+\q,15-\q)--(\rr-\q,15-\q) arc(270:360:\q) -- (\p,15);
}
{\nc\p{4} \nc\q{.4} \nc\rr{5} 
\fill[white] (\p,15) arc(180:270:\q) (\p+\q,15-\q)--(\rr-\q,15-\q) arc(270:360:\q) -- (\p,15);
}
{\nc\p{15+0} \nc\q{2} \nc\rr{15+9} 
\fill[red!20] (\p,15) arc(180:270:\q) (\p+\q,15-\q)--(\rr-\q,15-\q) arc(270:360:\q) -- (\p,15);
}
{\nc\p{15+4} \nc\q{.4} \nc\rr{15+5} 
\fill[white] (\p,15) arc(180:270:\q) (\p+\q,15-\q)--(\rr-\q,15-\q) arc(270:360:\q) -- (\p,15);
}
\node () at (4.5,13.8) {$\lam(\th)$};
\node () at (15+4.5,13.8) {$\rho(\phi)$};
\end{scope}
\begin{scope}[shift={(35,-15)}]
\fill[red!20] (0,0)--(24,0)--(14,15)--(10,15)--(0,0);
\draw[|-|] (26,15)--(26,0);
\node () at (26.6,7.5) {$\g$};
{\nc\p{0} \nc\q{2} \nc\rr{9} 
\fill[red!20] (\p,15) arc(180:270:\q) (\p+\q,15-\q)--(\rr-\q,15-\q) arc(270:360:\q) -- (\p,15);
}
{\nc\p{4} \nc\q{.4} \nc\rr{5} 
\fill[white] (\p,15) arc(180:270:\q) (\p+\q,15-\q)--(\rr-\q,15-\q) arc(270:360:\q) -- (\p,15);
}
{\nc\p{15+0} \nc\q{2} \nc\rr{15+9} 
\fill[red!20] (\p,15) arc(180:270:\q) (\p+\q,15-\q)--(\rr-\q,15-\q) arc(270:360:\q) -- (\p,15);
}
{\nc\p{15+4} \nc\q{.4} \nc\rr{15+5} 
\fill[white] (\p,15) arc(180:270:\q) (\p+\q,15-\q)--(\rr-\q,15-\q) arc(270:360:\q) -- (\p,15);
}
\node () at (4.5,13.8) {$\lam(\th)$};
\node () at (15+4.5,13.8) {$\rho(\phi)$};
\end{scope}
\begin{scope}
\fill[blue!20] (0,0)--(4,0)--(24,15)--(0,15)--(0,0);
{\nc\p5 \nc\q{1} \nc\rr{14} 
\fill[blue!20] (\p,0) arc(180:90:\q) (\p+\q,\q)--(\rr-\q,\q) arc(90:0:\q) -- (\p,0);
}
{\nc\p9 \nc\q{.3} \nc\rr{10} 
\fill[white] (\p,0) arc(180:90:\q) (\p+\q,\q)--(\rr-\q,\q) arc(90:0:\q) -- (\p,0);
}
\foreach \x in {0,1,2,3} {\lv{15+0.5*\x} \lv{20+0.5*\x}}
\foreach \x in {0,1} {\draw[dotted](17+5*\x,0)--(19+5*\x,0);}
\draw[|-|] (-1,15)--(-1,0);
\node () at (-1.6,7.5) {$\a$};
\end{scope}
\begin{scope}[shift={(35,0)}]
\fill[blue!20] (24-0,0)--(24-4,0)--(24-24,15)--(24-0,15)--(24-0,0);
{\nc\p{24-14} \nc\q{1} \nc\rr{24-5} 
\fill[blue!20] (\p,0) arc(180:90:\q) (\p+\q,\q)--(\rr-\q,\q) arc(90:0:\q) -- (\p,0);
}
{\nc\p{24-10} \nc\q{.3} \nc\rr{24-9} 
\fill[white] (\p,0) arc(180:90:\q) (\p+\q,\q)--(\rr-\q,\q) arc(90:0:\q) -- (\p,0);
}
\foreach \x in {0,1,2,3} {\lv{0+0.5*\x} \lv{5+0.5*\x}}
\foreach \x in {0,1} {\draw[dotted](2+5*\x,0)--(4+5*\x,0);}
\draw[|-|] (26,15)--(26,0);
\node () at (26.6,7.5) {$\b$};
\end{scope}
\end{tikzpicture}}
\vspace{-1em}
\caption{The products $\a\g=\th$ and $\b\g=\phi$.}
\label{fig:3}
\end{center}
\end{figure}

\section*{Acknowledgements}
This work was supported by the Engineering and Physical Sciences Research Council [EP/V002953/1, EP/V003224/1] and the Australian Research Council [FT190100632].  We thank John Truss for his help in proving Lemma \ref{lem:DBL}.

\end{document}